\documentclass[preprint]{elsarticle}
\usepackage{lineno}
\modulolinenumbers[5]

\journal{Int. J. Numer. Meth. Engng}

\usepackage{amsmath}
\usepackage{amssymb}
\usepackage{amsthm}
\usepackage{graphicx}
\usepackage{epic,eepic,epsfig}
\usepackage{color}
\usepackage{subfigure}  
\usepackage{placeins}   
\usepackage{multirow}
\usepackage{epstopdf}
\usepackage{varwidth}
\usepackage{float}
\usepackage[table]{xcolor}

\newtheorem{theorem}{Theorem}[section]
\newtheorem{example}{Example}[section]
\newtheorem{lemma}{Lemma}[section]
\newtheorem{remark}{Remark}[section]

\newtheorem{algorithm}{Algorithm}[section]

\definecolor{tabclr}{cmyk}{0,0,1,0}

\newcommand{\bx}{\boldsymbol{x}}

\newcommand{\balpha}{\boldsymbol{\alpha}}

\newcommand{\Rmnum}[1]{\uppercase\expandafter{\romannumeral #1}} 

\begin{document}

\title{Adaptive BDDC algorithms for the system arising from plane wave discretization of Helmholtz equations}

\author[XTU1]{Jie Peng\corref{cof}}
\ead{xtu\_pengjie@163.com}

\author[XTU1,XTU2]{Junxian Wang\corref{cof}}
\ead{wangjunxian@xtu.edu.cn}

\author[XTU1,XTU2]{Shi Shu\corref{cor}}
\ead{shushi@xtu.edu.cn}

\cortext[cof]{These authors contributed equally to this work and should be considered co-first authors}
\cortext[cor]{Corresponding author}
\address[XTU1]{School of Mathematics and Computational Science, Xiangtan University, Xiangtan 411105, China}
\address[XTU2]{Hunan Key Laboratory for Computation and Simulation in Science and Engineering,
Xiangtan University, Xiangtan 411105, China}

\begin{abstract}
 Balancing domain decomposition by constraints (BDDC) algorithms with adaptive primal constraints are developed in a concise variational framework for the weighted plane wave least-squares (PWLS) discritization of Helmholtz equations with high and various wave numbers.
 The unknowns to be solved in this preconditioned system are defined on elements rather than vertices or edges,
 which are different from the well-known discritizations such as the classical finite element method.
 Through choosing suitable ``interface" and appropriate primal constraints with complex coefficients and  introducing some local techniques, we developed a two-level adaptive BDDC algorithm for the PWLS discretization,
 and the condition number of the preconditioned system is proved to be bounded above by a user-defined tolerance and a constant which is only dependent on the maximum number of interfaces per subdomain. A multilevel algorithm is also attempted to resolve the bottleneck in large scale coarse problem.
 Numerical results are carried out to confirm the theoretical results and illustrate the efficiency of the proposed algorithms.
\end{abstract}


\begin{keyword}
Helmholtz equation, high wave number, plane wave discretization, BDDC algorithm, adaptive primal constraints
\MSC[2010] 65N30 \sep 65F10 \sep 65N55
\end{keyword}

\maketitle

\section{Introduction}\label{sec:1}\setcounter{equation}{0}

Helmholtz equations have many applications in  electromagnetic radiation, acoustics scattering and exploration seismology.
%
%
%
As the oscillatory behavior of the solution of the Helmholtz equation,
the corresponding discrete system is usually huge and highly indefinite, especially for high wave numbers.
The plane wave methods, which fall into the class of Trefftz methods \cite{HMP2016},
are popular discretization methods for solving this kind of equations \cite{FHH2003,FWT2004,CD1998,HKM2004,L1996,LA2001}.
Compared with the classical finite element method (FEM) \cite{BS1997, DB1999}, the plane wave methods can significantly
reduce the required degree of freedom under the same error precision, and with the increase of the wave number, the superiority is more obvious.
The weighted plane wave least-squares method (PWLS) is a frequently-used plane wave method \cite{TP1995,S1998,MW1999,HY2014,HL20171,HL20172}.
One advantage of PWLS over the other plane wave methods is that the stiffness matrix of the PWLS discrete system is Hermitian positive definite, this
lead to solve the resulting system by preconditioned conjugate gradient (PCG) method, and
the preconditioner plays an important role in the iterative process.

The development of an efficient solver or preconditioner for the Helmholtz equation has led to a great interest over the course
of the past decades \cite{CW1992}.
Domain decomposition (DD) methods are powerful parallel methods for solving the systems arising from finite element discretization of elliptic problems.
There exist many well known nonoverlapping DD methods for solving indefinite systems of Helmholtz equations, like the Robin-type DD method \cite{LX2014,CLX2016},
the substructuring method \cite{HZ2016}, the finite element tearing and interconnecting (FETI) method \cite{FMT1999,TMF2001} and the dual-primal finite element tearing and interconnecting (FETI-DP) method \cite{M2002,FAT2005}.
Alternative advanced nonoverlapping DD method is
the balancing domain decomposition by constraints (BDDC) methods \cite{Dohrmann2003,LW2006,BS2007}.
In the works by Li and Tu \cite{LT2008,TL2009}, the BDDC, which incorporated some plane waves in the coarse problem to accelerate the convergence rate,  were extended to solving the FEM discrete system of Helmholtz equation. Numerical experiments illustrate that the convergence rate depends on a logarithmic pattern of the dimension of the local subdomain problems, improves with the decrease of the subdomain diameters, and depends on the wave number but it can be improved by including more plane wave continuity constraints in the coarse space.
Therefore, to enhance the robustness of the BDDC methods for solving the Helmholtz systems, the selection of good primal constraints should be necessary.

The main objective of this paper is to propose an adaptive BDDC preconditioner.
Adaptive BDDC preconditioner is an advanced BDDC method using a transformation of basis, the
primal unknowns are always selected by solving some generalized eigenvalue problems with respect to the local problems, and adaptively depends on a given tolerance \cite{MS2007,DP2012,PD2013,DSO2015,KC2015,KCW2017,Z2016,KRR2016,JGC2016,PC2016,KCX2017,DV2017,PSW2017}.
Since these local problems can indicate the bad behavior of the standard coarse problem,
they can be used to select the primal constraints to enhance the convergence of the iteration \cite{KCW2017}.
%
%
However, there is an undeniable fact that the number of primal unknowns increases as the number of subdomains increases,
the corresponding coarse problem will become too large and hard to solve directly.
This leads to multilevel extension of this algorithm naturally \cite{Dohrmann2003,T20071,T20072,MSD2008,ZT2017}.

In this paper, we will develop an adaptive BDDC preconditioner. To be more specific,
we will extend the existing methods in \cite{KCW2017,KCX2017,PSW2017} to  PWLS discretization of Helmholtz equation with high and various wave numbers,
and present a complete theory.
Contrast to the classical FEM, the dofs in the PWLS method are defined on elements rather than vertices or edges,
we thus introduce a kind of special ``interface",
which is different from the existing nonoverlapping DD methods.
%
%
%
Since the PWLS discrete system consists of complex coefficients, we construct our transformation operators by using a series of local generalized eigenvalue problems with respect to the parallel sum and the primal constraints are formed by the eigenvectors with their complex modulus of eigenvalues greater than a given tolerance $\Theta$, which is different from \cite{KCW2017,KCX2017,PSW2017}.
As the spectral condition number of the PWLS discretizations of the Helmholtz equations with high wave numbers
grows with the increase of the number of plane wave bases in each element and the decrease of the grid size \cite{HMP2016},
some local techniques are introduced to overcome this difficulty.
Then, by introducing some other auxiliary spaces and operators, we arrive at our two-level adaptive BDDC algorithm in variational framework for PWLS discretizations.
The condition number bound of the two-level adaptive BDDC preconditioned systems, $C\Theta$, can be derived by using the properties of the auxiliary spaces and involved operators, where $C$ is a constant which depends only on the maximum number of interfaces per subdomain. Compared with the previous work for mortar discretizations in \cite{PSW2017}, the variational framework in this paper is more concise.

We perform numerical experiments for various model problems.
These results verify the correctness of theoretical results, and show that our two-level adaptive BDDC algorithms are scalability with respect to the angular frequency, the number of subdomains and mesh size.
It is worth pointing out that the algorithm with deluxe scaling matrices has more advantages over the algorithm with multiplicity scaling matrices in the size of coarse problem even for model problem with constant medium, which is different from the adaptive BDDC algorithms for the two-order elliptic problems.


However, since the number of primal unknowns increase as the wave number or the number of subdomains increase,
we attempt to construct a multilevel adaptive BDDC algorithm to resolve the bottleneck in solving large scale coarse problem.
The coarser subdomains are gathered by a certain amount of subdomains at the finer level, and the new ``interface" can be obtained naturally.
Numerical results show that the multilevel algorithm can reduce the size of the coarse problem, and it is also robust to solve the Helmholtz equations with high wave number.

The rest of this paper is organized as follows. In Section 2, The PWLS formulation will be presented for a Helmholtz equation with Robin boundary condition.
In Section 3, we will firstly derive the Schur complete variational problem and its corresponding function space, introduce some auxiliary spaces and dual-primal basis functions further, and then illustrate our two-level and multilevel adaptive BDDC algorithms. An estimate of the condition numbers will be analyzed in Section 4, and various numerical experiments are presented to verify the performance of our algorithm in Section 5.
Finally, we will give a conclusion in Section 6.

\section{Weighted plane wave least squares formulation}\label{sec:2}
\setcounter{equation}{0}


\subsection{Model problem}

Let $\Omega \in \mathbb{R}^2$ be a bounded and connected Lipschitz domain with boundary $\partial \Omega$.
Consider the Helmholtz equation with Robin boundary condition (\cite{HY2014,HL20171})
\begin{eqnarray}\label{model equation}
\left\{
\begin{array}{rcll}
 -\Delta u  - \kappa^2 u &=&  0        &in~ \Omega,\\
 (\partial_{\bf n} + i \kappa )u &=&  g    &    on~ \partial \Omega,
\end{array}
\right.
\end{eqnarray}
where $i = \sqrt{-1}$, $\partial_{\bf n}$ and $\kappa$ are separately the imaginary unit, the outer normal derivative and the wave number, $g \in L^2(\Omega)$.
The wave number $\kappa = \omega/c > 0$, where $\omega$ and $c$ are separately called the angular frequency and the wave speed.
The above problem is usually seen as an approximation of the acoustic scattering problem, and the wave speed $c$ (and hence $\kappa$)  can be a constant or variable function. For $g = 0$, the second equation of \eqref{model equation} becomes a general representation of an absorbing boundary condition \cite{FWT2004}.

\subsection{Weighted plane wave least squares discretization}

Following Hu and Zhang \cite{HZ2016}, we first define a quadrilateral mesh $\mathcal{T}_h$, namely, dividing $\Omega$ into
$$
\bar{\Omega} = \bigcup\limits_{k=1}^{N_h} \bar{\Omega}_k,
$$
where the quadrilateral elements $\{\Omega_k\}$ satisfy that $\Omega_m \cap \Omega_l = \emptyset, m \neq l$, $h_k$ is the size of $\Omega_k$ and $h = \max\limits_{1\le k \le N_h} h_k$.
Define
\begin{equation*}
\begin{array}{l}
\gamma_{kj} = \partial \Omega_k \cap \partial \Omega_j,~~ \mbox{for}~k,j = 1,\cdots,N_h ~\mbox{and}~ k\neq j, \\
\gamma_k = \partial \Omega_k \cap \partial \Omega,~~\mbox{for}~k=1,\cdots,N_h,\\
\mathcal{F}_B = \bigcup\limits_{k=1}^{N_h} \gamma_k,~~\mathcal{F}_I = \bigcup\limits_{k \neq j} \gamma_{kj}.
\end{array}
\end{equation*}

Throughout this paper, we assume that each
$\kappa_k:= \kappa|_{\Omega_k}$ is a constant. Let $V(\Omega_k)$ be the local space whose members satisfy the
homogeneous Helmholtz's equation \eqref{model equation} on $\Omega_k$:
\begin{equation*}
V(\Omega_k) = \{v_k \in H^1({\Omega_k}):~\Delta v_k + \kappa_k^2 v_k = 0\},~k=1,\cdots,N_h.
\end{equation*}
Define the global space
$$
V({\mathcal{T}_h}) = \bigcup_{k=1}^{N_h} V(\Omega_k).
$$

As we all know,
problem \eqref{model equation} to be solved is equivalent to find the local solution
$u_k: = u|_{\Omega_k} \in \{v \in H^1({\Omega_k}):~\nabla v \in H(div;\Omega_k)\}$
such that
\begin{eqnarray}\label{submodel}
\left\{
\begin{array}{rclll}
    -\Delta u_k  -\kappa_k^2 u_k &=&  0    & in ~\Omega_k, & \\
     (\partial_{\bf n} + i\kappa_k )u_k&=&  g    & on~ \gamma_k, &
\end{array}
\right. k=1,2,\cdots,N_h,
\end{eqnarray}
with the continuity conditions for $u$ and its normal derivative on the interfaces between the
elements:
\begin{eqnarray}\label{interface condition}
u_k  - u_j  =  0,~ \partial_{{\bf n}_k}u_k +  \partial_{{\bf n}_j}u_j =   0, ~\mbox{on}~\gamma_{kj}, ~k,j = 1,\cdots,N_h~and~k\neq j.
\end{eqnarray}

In the weighted plane wave least squares (PWLS) formulation, a finite dimensional subspace of  $V({\mathcal T_h})$ is introduced,
\begin{eqnarray}\label{Vp-space}
V_p({\mathcal T_h}) =span\{\varphi_{m,l}:~1 \le l \le p, 1\le m \le N_h\},
\end{eqnarray}
where
\begin{equation*}
\varphi_{m,l}(\bx)=
\left\{
\begin{array}{ll}
y_{m,l}(\bx)     &\bx \in \bar{\Omega}_m\\
     0          &\bx \in \Omega \backslash \bar{\Omega}_m
\end{array}
\right. ,
\end{equation*}
here $y_{m,l} (l=1,\cdots,p)$ denote the wave shape functions on $\Omega_m$, which satisfy
\begin{equation*}
\left\{
\begin{array}{rcl}
 y_{m,l}(\bx) &=& e^{i\kappa(\bx \cdot\balpha_l)},~\bx \in \bar{\Omega}_m,\\
 |\balpha_l| &=& 1, \\
   \balpha_l &\neq& \balpha_s,~\mbox{for}~l \neq s,
\end{array}
\right.
\end{equation*}
and $\balpha_l$ $(l = 1,\cdots,p)$ are unit wave propagation directions.
In particular, during numerical simulations, we set
\begin{equation*}
\balpha_l = \left(
  \begin{array}{c}
    \cos(2\pi(l-1)/p)\\
    \sin(2\pi(l-1)/p)\\
  \end{array}
\right).
\end{equation*}

By using the plane wave finite dimensional space $V_p(\mathcal{T}_h)$ defined above, the PWLS formulation associated with problem \eqref{submodel} and \eqref{interface condition} can be described as follows: find $u \in V_p({\mathcal T_h})$ such that
\begin{eqnarray}\label{115-dis}
a(u,v) =  \mathcal{L}(v),~\forall v \in V_p({\mathcal T_h}),
\end{eqnarray}
where
\begin{eqnarray}\nonumber
  a(u,v) &=& \sum_{j\neq k} (\alpha_{kj}\int_{\gamma_{kj}} (u_k-u_j)\cdot\overline{(v_k-v_j)}ds
+\beta_{kj}\int_{\gamma_{kj}}
({\partial_{{\bf n}_k}u_k+\partial_{{\bf n}_j}u_j})\cdot\overline{(\partial_{{\bf n}_k}v_k+\partial_{{\bf n}_j}v_j)}ds)
  \\ \label{115-1}
   && + \sum_{k=1}^{N_h} \theta_k \int_{\gamma_k} ((\partial_{\bf n}+i\kappa_k ){u}_k) \cdot \overline{(\partial_{\bf n}+i \kappa_k){v}_k}ds,\\\label{115-2}
\mathcal{L}(v) &=& \sum_{k=1}^{N_h} \theta_k \int_{\gamma_k} g\cdot\overline{(\partial_{\bf n}+i \kappa_k ){v}_k}ds,
\end{eqnarray}
here $\overline{\diamond}$ denotes the complex conjugate of the complex quantity $\diamond$, the Lagrange multipliers
$$\alpha_{kj} = h^{-1} + \kappa_{kj},~\beta_{kj} = h^{-1} \kappa_{kj}^{-2} + \kappa_{kj}^{-1}~ \mbox{with}~\kappa_{kj} = (\kappa_k + \kappa_j)/2,$$
and
$$\theta_k = h^{-1} \kappa_k^{-2} + \kappa_k^{-1}.$$

This discrete variational problem \eqref{115-dis} is derived by the minimization of a quadratic functional,
 and the basic idea of the minimization problem is to find a function in $V_p(\mathcal{T}_h)$ so that it can satisfy the external boundary conditions and the interface conditions as far as possible \cite{HZ2016}. From Theorem 3.1 of \cite{HY2014}, we can see that the continuous variational problem associate with \eqref{115-dis} is equivalent to the reference problem \eqref{submodel} and \eqref{interface condition}.

It is clear that $a(\cdot,\cdot)$ is sesquilinear and Hermitian, and similar to  the proof of Theorem 3.1 in \cite{HY2014}, we can see that $a(v,v) \ge 0$, and $a(v,v) = 0$ for $v \in V(\mathcal{T}_h)$ if and only if $v = 0$. Therefore,  $a(\cdot,\cdot)$ is Hermitian positive definite (HPD).

Due to the linear system obtained in \eqref{115-dis} are large and highly ill-conditioned when the
wave number is large, it's necessary to study a fast solver for this system.
Adaptive BDDC algorithm is a novel domain decomposition (DD) method with enriched coarse spaces \cite{MS2007},
this algorithm has been successfully applied to solve discrete systems obtained by various discretization methods, such as conforming Galerkin \cite{KC2015},
discontinuous Galerkin \cite{KCX2017}, and mortar methods \cite{PSW2017} and so on.  However, adaptive BDDC
algorithm for PWLS discretizations has not previously been discussed in the literature.
Here we will extend the adaptive BDDC algorithm to PWLS discretizations with high and various wave numbers.



\section{Adaptive BDDC preconditioner}\setcounter{equation}{0}

\subsection{Domain decomposition and Schur complement problem}

Differ from the discretizations which dofs are defined on the vertices or edges of the mesh,
the dofs in the PWLS discretization are defined on the elements,
therefore, we need to introduce a special interface and domain decomposition firstly.

Let $\{D_r\}_{r=1}^{N_d}$ be a non-overlapping subdomain partiaon of $\Omega$ and each $D_r$ consists of several complete elements and part of the elements in $\mathcal{T}_h$ (see
Figure \ref{fig-dd-1}).
Let $\mathcal{T}_d$ denote the coarse partition associated with the subdomains $D_1, D_2, \cdots, D_{N_d}$.

\begin{figure}[H]
  \centering
  \includegraphics[width=0.3\textwidth]{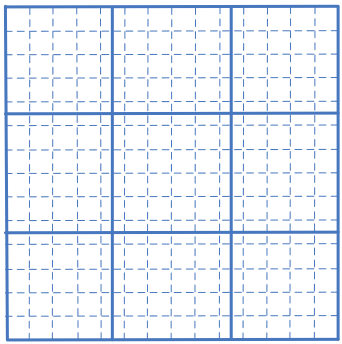}\\
  \caption{Element: the small square with dotted line boundary, subdomain: the square
with solid line boundary.}\label{fig-dd-1}
\end{figure}

For any given $r = 1,\cdots, N_d$, define
\begin{align*}
\mathcal{F}_{I}^{(r)} = \{\tilde{\gamma}_{kj}:~ \tilde{\gamma}_{kj} = \gamma_{kj}|_{\bar{D}_r},~ \forall \gamma_{kj} \in \mathcal{F}_I\},~
\mathcal{F}_{B}^{(r)} = \{\tilde{\gamma}_k:~ \tilde{\gamma}_k = \gamma_{k}|_{\bar{D}_r},~ \forall \gamma_{k} \in \mathcal{F}_B\}.
\end{align*}
From this definition, the sesquilinear form $a(\cdot,\cdot)$ defined in \eqref{115-1} can be rewrite as
\begin{eqnarray}\label{115-1-0}
  a(u,v) = \sum\limits_{r=1}^{N_d} a_r (u,v),
\end{eqnarray}
where
\begin{eqnarray}\nonumber
a_r(u,v) &=& \sum_{\tilde{\gamma}_{kj} \in \mathcal{F}_I^{(r)}} (\alpha_{kj}\int_{\tilde{\gamma}_{kj}} (u_k-u_j) \cdot \overline{(v_k-v_j)}ds + \beta_{kj}\int_{\tilde{\gamma}_{kj}}
({\partial_{{\bf n}_k}u_k+\partial_{{\bf n}_j}u_j}) \cdot \overline{(\partial_{{\bf n}_k}v_k+\partial_{{\bf n}_j}v_j)}ds)
  \\ \label{def-AUV}
   &&+ \sum_{\tilde{\gamma}_k \in \mathcal{F}_B^{(r)}} \theta_k \int_{\tilde{\gamma}_k} ((\partial_{\bf n}+iw ){u}_k) \cdot\overline{(\partial_{\bf n}+i \omega){v}_k}ds.
\end{eqnarray}
and it is easy to verify that $a_r(\cdot,\cdot)$ is Hermitian positive semi-definite.

If $\partial D_r \cap \partial D_j(r \neq j)$ is a common edge, we call it an interface $\Gamma_{rj}$ (or $\Gamma_{jr}$).
Let $\Gamma=\cup\Gamma_{rj}$, $N_f$ denotes the number of interfaces.
Since there is one-to-one correspondence between
any given $\Gamma_{rj}$ and the interface, we denote $\{\Gamma_{rj}\}$ by $\{F_k: k=1,\cdots,N_f\}$ for convenience.
Denote $\{v_k: k=1,\cdots,N_v\}$ be the set of the vertices corresponding to $\mathcal{T}_d$.
For each $D_r (r = 1,\cdots,N_d)$, let
\begin{eqnarray}\label{def-M-i}
&& \mathcal{M}_r := \{k:~F_k \subset \partial D_r \backslash \partial \Omega,~\mbox{for}~1\le k \le N_f\},
\\\label{def-M-i-v}
&& \mathcal{M}_r^c := \{k:~v_k \subset \partial D_r \backslash \partial \Omega,~\mbox{for}~1\le k \le N_v\}.
\end{eqnarray}
For sake of argument, we set the elements of $\mathcal{M}_r$ as $\{r_m:~1\le m \le f_r\}$, where $f_r$ denotes the size of $\mathcal{M}_r$, i.e.
the number of interfaces on $\partial \Omega_r \backslash \partial \Omega$.

Set
\begin{align*}
\mathcal{D}_r &:= \{m:~\Omega_m \cap D_r \neq \emptyset,~\mbox{for}~1 \le m \le N_h\},~1 \le r \le N_d,\\
\mathcal{I}_r &:= \{m:~\Omega_m \subset D_r,~\mbox{for}~1 \le m \le N_h\},~1 \le r \le N_d,\\
\mathcal{V}_k &:= \{m:~\Omega_m \cap v_k \neq \emptyset,~\mbox{for}~1 \le m \le N_h\},~1 \le k \le N_v,\\
\mathcal{F}_k &:= \{m:~\Omega_m \cap F_k \neq \emptyset,~\mbox{for}~1 \le m \le N_h~\mbox{and}~m \notin \mathcal{V}_l,~1 \le l \le N_v\},~1 \le k \le N_f.
\end{align*}
and (see Figure \ref{fig-Dv-DF} for illustration)
\begin{align*}
D_{v_k} = \bigcup\limits_{m \in \mathcal{V}_k} \Omega_m~\mbox{and}~D_{F_k} = \bigcup\limits_{m \in \mathcal{F}_k} \Omega_m.
\end{align*}
Then the special interface subdomain can be defined as
\begin{align*}
D_{\Gamma} = \left(\bigcup \limits_{k=1}^{N_v} D_{v_k}\right) \bigcup \left(\bigcup \limits_{k=1}^{N_f} D_{F_k}\right).
\end{align*}
\begin{figure}[H]
  \centering
  \includegraphics[width=0.3\textwidth]{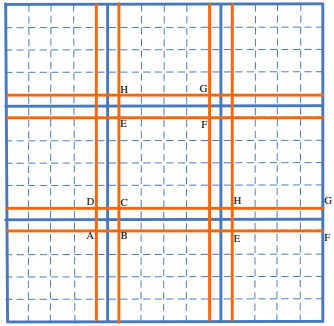}\\
  \caption{The rectangle ABCD and EFGH are separately denote the subdomains $D_{v_k}$ and $D_{F_k}$.}\label{fig-Dv-DF}
\end{figure}

Based on the aforementioned domain decomposition, the Schur complement problem of \eqref{115-dis} can be introduced.

Let the local spaces
\begin{eqnarray}\label{def-V-1}
&& V_I = \oplus_{r=1}^{N_d}V^{(r)}_I, ~V^{(r)}_I = span\{\varphi_{m,l}:~1 \le l \le p,~\forall m \in \mathcal{I}_r\},~1 \le r \le N_d,\\\nonumber
&& V_{k,c} = span\{\varphi_{m,l}:~1 \le l \le p,~\forall m\in \mathcal{V}_k\},~1 \le k \le N_v,
\end{eqnarray}
and
\begin{align*}
V_k = span\{\varphi_{m,l}:~1 \le l \le p,~\forall m \in \mathcal{F}_k\},~1 \le k \le N_f.
\end{align*}

For any interface $F_k$, we always assume that it is a common edge of two subdomains $D_r$ and $D_j$.
Let $n_k = p |\mathcal{F}_k|$, where $|\diamond|$ denotes the size of set $\diamond$. Define a vector
\begin{eqnarray}\label{def-Phik}
\Phi^k = (\phi^{k}_{1},\cdots,\phi^{k}_{n_k})^T:=(\phi^{k}_{m,1},\cdots,\phi^{k}_{m,p}, \forall m \in \mathcal{F}_k)^T,
\end{eqnarray}
where $\phi^{k}_{m,l} \in V_I^{(r)} \oplus V_I^{(j)}\oplus V_k (1 \le l \le p, \forall m\in \mathcal{F}_k)$ satisfy
\begin{eqnarray}\label{def-Hw}
\left\{\begin{array}{l}
a(\phi^{k}_{m,l}, v) = 0, ~~\forall v \in V_I,\\
\phi^{k}_{m,l}|_{\bar{D}_{F_k}} = \varphi_{m,l}|_{\bar{D}_{F_k}}.
\end{array}\right.
\end{eqnarray}

For any vertex $v_k$, let
\begin{eqnarray}\label{def-mathcalN-k}
\mathcal{N}_k = \{r:~\partial D_r \cap v_k \neq \emptyset, 1\le r\le N_d\}, ~n_k^c = p|\mathcal{V}_k|,
\end{eqnarray}
we can define another vector
\begin{eqnarray}\label{def-Phik-v}
\Psi^{k} = (\psi^{k}_{1},\cdots,\psi^{k}_{n_k^c})^T:=(\psi^{k}_{m,1},\cdots,\psi^{k}_{m,p}, \forall m\in \mathcal{V}_k)^T,
\end{eqnarray}
where $\psi^{k}_{m,l} \in (\oplus_{r \in \mathcal{N}_k} V_I^{(r)}) \oplus V_{k,c} (1 \le l \le p, \forall m\in \mathcal{V}_k)$ satisfy
\begin{eqnarray}\label{def-Hw-v}
\left\{\begin{array}{l}
a(\psi^{k}_{m,l}, v) = 0, ~~\forall v \in V_I,\\
\psi^{k}_{m,l}|_{\bar{D}_{v_k}} = \varphi_{m,l}|_{\bar{D}_{v_k}}.
\end{array}\right.
\end{eqnarray}


Utilizing these two vectors, the function space related to the Schur complement problem can be defined as
\begin{align*}
\hat{W} = (\oplus_{k=1}^{N_f} W_{k}) \oplus (\oplus_{k=1}^{N_v} W_{k,c}),
\end{align*}
where
\begin{align*}
W_{k} = span\{\phi^{k}_{1},\cdots,\phi^{k}_{n_k}\}~\mbox{and}~W_{k,c} = span\{\psi^{k}_{1},\cdots,\psi^{k}_{n_k^c}\}.
\end{align*}

Let $\hat{S}: \hat{W} \rightarrow \hat{W}$ be the Schur complement operator defined by
\begin{eqnarray}\label{def-operator-hat-S}
(\hat{S} \hat{u}, \hat{v}) = a(\hat{u}, \hat{v}),~~\forall \hat{u}, \hat{v} \in \hat{W}.
\end{eqnarray}
Due to the coercive of the restriction of $a(\cdot,\cdot)$ on $\hat{W}$, $\hat{S}$ is a HPD operator.

Then the Schur complement variational form of \eqref{115-dis} can be expressed as:
find $\hat{w} \in \hat{W}$ such that
\begin{eqnarray}\label{def-schur-system}
(\hat{S} \hat{w}, \hat{v}) = \mathcal{L}(\hat{v}),~~\forall \hat{v} \in \hat{W}.
\end{eqnarray}

In order to give an adaptive BDDC preconditioner for solving the Schur complement problem \eqref{def-schur-system},
we will give some preparations in next subsection.

\subsection{Some auxiliary spaces and dual-primal bases}\label{sec:3}

The BDDC algorithm consist of several independent subdomain problems and one global coarse problem.
The function spaces associated with these problems are spanned by the so-called dual basis functions and primal basis functions.
In the following, we will introduce some auxiliary spaces and dual-primal bases.

For each subdomain $D_r$, let truncated basis functions
\begin{equation*}
\varphi_{m,l}^{(r)}(\bx)=
\left\{
\begin{array}{ll}
\varphi_{m,l}(\bx)    &\bx \in \bar{\Omega}_m \cap \bar{D}_r\\
     0          &\bx \in \Omega \backslash (\bar{\Omega}_m \cap \bar{D}_r)
\end{array}
\right. , ~1\le l \le p, m \in \mathcal{D}_r.
\end{equation*}
Define the corresponding local spaces
\begin{align*}
& V_{k,c}^{(r)} = span\{\varphi_{m,l}^{(r)}:~1 \le l \le p,~\forall m \in \mathcal{V}_k\},~\mbox{for}~k \in \mathcal{M}_r^c,\\
& V_k^{(r)} = span\{\varphi_{m,l}^{(r)}:~1 \le l \le p,~\forall m \in \mathcal{F}_k\},~\mbox{for}~k \in \mathcal{M}_r,
\end{align*}
and
\begin{align*}
V^{(r)} = V^{(r)}_I \oplus (\oplus_{k\in \mathcal{M}_r} V_k^{(r)}) \oplus (\oplus_{k\in \mathcal{M}_r^c} V_{k,c}^{(r)}),
\end{align*}
where $V^{(r)}_I$ is defined in \eqref{def-V-1}.

Using the above-mentioned local spaces, some new set of basis functions and corresponding spaces can be constructed firstly.

For any given interface $F_k = \partial D_r \cap \partial D_j$, denote
\begin{eqnarray}\label{def-Phi-Psi-k-nu}
\Phi^{k,\nu} =  (\phi^{k,\nu}_{m,1}, \cdots, \phi^{k,\nu}_{m,p}, \forall m\in \mathcal{F}_k)^T~\mbox{and}~\bar{\Phi}^{k,\nu} = (\bar{\phi}^{k,\nu}_{m,1}, \cdots, \bar{\phi}^{k,\nu}_{m,p}, \forall m\in \mathcal{F}_k)^T,~\nu=r,j,
\end{eqnarray}
where $\phi^{k,\nu}_{m,l} \in V_I^{(\nu)} \oplus V_k^{(\nu)}$ and $\bar{\phi}^{k,\nu}_{m,l} \in V^{(\nu)}$ separately satisfy
\begin{eqnarray}\label{def-Hw-i}
\left\{\begin{array}{l}
a_{\nu}(\phi^{k,\nu}_{m,l}, v) = 0, ~~\forall v \in V^{(\nu)}_I\\
\phi^{k,\nu}_{m,l}|_{\bar{D}_{F_k}} = \varphi^{(\nu)}_{m,l}|_{\bar{D}_{F_k}}
\end{array}\right.,~~1 \le l \le p, m \in \mathcal{F}_k,
\end{eqnarray}
and
\begin{eqnarray}\label{def-barHw-i-Fim}
\left\{\begin{array}{l}
a_{\nu}(\bar{\phi}^{k,\nu}_{m,l}, v) = 0, ~~\forall v \in V^{(\nu)} \backslash V_k^{(\nu)}\\
\bar{\phi}^{k,\nu}_{m,l}|_{\bar{D}_{F_k}} = \varphi^{(\nu)}_{m,l}|_{\bar{D}_{F_k}}
\end{array}
\right.,~~1 \le l \le p, m \in \mathcal{F}_k.
\end{eqnarray}

For any given vertex $v_k$, denote
\begin{eqnarray}\label{def-Psi-k-nu-c}
\Psi^{k,\nu} = (\psi^{k,\nu}_{m,1}, \cdots, \psi^{k,\nu}_{m,p}, \forall m\in \mathcal{V}_k)^T,~\nu \in \mathcal{N}_k,
\end{eqnarray}
where $\psi^{k,\nu}_{m,l} \in V_I^{(\nu)} \oplus V_{k,c}^{(\nu)}$ satisfy
\begin{eqnarray}\label{def-Hw-psi-c}
\left\{\begin{array}{l}
a_{\nu}(\psi^{k,\nu}_{m,l}, v) = 0, ~~\forall v \in V^{(\nu)}_I\\
\psi^{k,\nu}_{m,l}|_{\bar{D}_{v_k}} = \varphi^{(\nu)}_{m,l}|_{\bar{D}_{v_k}}
\end{array}\right.,~~1 \le l \le p, m \in \mathcal{V}_k.
\end{eqnarray}
and $\mathcal{N}_k$ is defined in \eqref{def-mathcalN-k}.

From the following lemma, we know that the basis functions $\phi^{k,\nu}_{m,l}$ and $\bar{\phi}^{k,\nu}_{m,l}$ $(1\le l \le p, \forall m \in \mathcal{F}_k)$ are available.
\begin{lemma}\label{lemma-ar}
The restriction of $a_r(\cdot,\cdot)$ $(r=1,\cdots,N_d)$ on $V^{(\nu)} \backslash V_k^{(\nu)}$ or $V_I^{(\nu)}$ is HPD.
\end{lemma}

\begin{proof}
First, from the definition \eqref{def-AUV} of $a_r(\cdot,\cdot)$, it's easy to verify that $a_r(u,u) \ge 0$ for any $u \in U$ ($U = V^{(\nu)} \backslash V_k^{(\nu)} ~\mbox{or}~V_I^{(\nu)}$).

Next, we show $a_r(u,u) = 0$ holds if and only if $u=0$. For simplicity, we only give the proof of the case $U = V^{(\nu)} \backslash V_k^{(\nu)}$,
it can be proved in a similar way if $U = V^{(\nu)}_I$. 

In fact, $a_r(u,u) = 0$ is equivalent to
\begin{align*}
& \sum_{\tilde{\gamma}_{mj} \in \mathcal{F}_I^{(r)}} (\alpha_{mj}\int_{\tilde{\gamma}_{mj}} |u_m-u_j|^2 ds
+ \beta_{mj}\int_{\tilde{\gamma}_{mj}} |{\partial_{{\bf n}_m}u_m+\partial_{{\bf n}_j}u_j}|^2 ds) \\
& + \sum_{\tilde{\gamma}_m \in \mathcal{F}_B^{(r)}} \theta_m \int_{\tilde{\gamma}_m} |(\partial_{\bf n}+iw ){u}_m|^2 ds = 0.
\end{align*}
Note that $\alpha_{mj}, \beta_{mj}, \theta_m >0$, the above equality implies that
\begin{align*}
\int_{\tilde{\gamma}_{mj}} |u_m-u_j|^2 ds = 0,~\int_{\tilde{\gamma}_{mj}} |{\partial_{{\bf n}_m}u_m+\partial_{{\bf n}_j}u_j}|^2 ds = 0,~\int_{\tilde{\gamma}_m} |(\partial_{\bf n}+iw ){u}_m|^2 ds = 0.
\end{align*}

From the definition of space $V^{(\nu)} \backslash V_k^{(\nu)}$, we have
\begin{align*}
u|_{D_{F_k}} = 0,~\forall u \in V^{(\nu)} \backslash V_k^{(\nu)}.
\end{align*}
It implies that if $\tilde{\gamma}_{mj} \subset \partial D_{F_k}$ and  $m \in \mathcal{F}_k$, we have
\begin{align*}
\int_{\tilde{\gamma}_{mj}} |u_m - u_j|^2 ds = 0~\Leftrightarrow~ \int_{\tilde{\gamma}_{mj}} |u_j|^2 ds = 0.
\end{align*}
Namely,
\begin{align*}
u_j = \sum\limits_{l=1}^p u_{j,l} \varphi_{j,l} = 0,~\mbox{on}~\tilde{\gamma}_{mj},
\end{align*}
where $u_{j,l} \in \mathbb{C} (l=1,\cdots,p)$.

From this and based on the linear independence of basis functions $\varphi_{j,l}(l=1,\cdots,p)$, we obtain
\begin{align*}
u_{j,l} = 0,~l=1,\cdots,p~\Leftrightarrow~u_j = 0,~\mbox{in}~\Omega_j,
\end{align*}
where $j$ satisfies $\tilde{\gamma}_{mj} \subset \partial D_{F_k} \backslash \partial D_r$ and $m \in \mathcal{F}_k$.

Furthermore, applying this process recursively, we can prove that $u = 0~\mbox{in}~D_r$.

Therefore we complete the proof of the coercive of $a_r(\cdot,\cdot)$ in $V^{(\nu)} \backslash V_k^{(\nu)}$.
\end{proof}

From \eqref{def-Hw}, \eqref{def-Hw-v}, \eqref{def-Hw-i} and \eqref{def-Hw-psi-c}, we can easily verify that
\begin{eqnarray}\label{rel-phik-phikv}
\phi^{k,\nu}_{m,l}|_{\bar{D}_{\nu}} = \phi^{k}_{m,l}|_{\bar{D}_{\nu}}
~\mbox{and}~
\psi^{k,\nu}_{m,l}|_{\bar{D}_{\nu}} = \psi^{k}_{m,l}|_{\bar{D}_{\nu}}.
\end{eqnarray}

Similar to $\Phi^k$ and $\Psi^k$ expressed in \eqref{def-Phik} and \eqref{def-Phik-v}, we denote
\begin{align*}
\Phi^{k,\nu} = (\phi^{k,\nu}_{1}, \cdots, \phi^{k,\nu}_{n_k})^T,~\bar{\Phi}^{k,\nu} = (\bar{\phi}^{k,\nu}_{1}, \cdots, \bar{\phi}^{k,\nu}_{n_k})^T,~\nu=r,j,
\end{align*}
and
\begin{align*}
\Psi^{k,\nu} = (\psi^{k,\nu}_{1}, \cdots, \psi^{k,\nu}_{n_k^c})^T,~\nu \in \mathcal{N}_k.
\end{align*}

Using the basis functions defined above, denote the auxiliary spaces
\begin{eqnarray}\label{def-W-k-nu}
 W_{k}^{(\nu)}=span\{\phi^{k,\nu}_{1}, \cdots, \phi^{k,\nu}_{n_k}\},~~
 \bar{W}_{k}^{(\nu)}=span\{\bar{\phi}^{k,\nu}_{1}, \cdots, \bar{\phi}^{k,\nu}_{n_k}\},~~\nu=r,j,
\end{eqnarray}
and
\begin{align*}
 W_{k,c}^{(\nu)}=span\{\psi^{k,\nu}_{1}, \cdots, \psi^{k,\nu}_{n_k^c}\},~~\forall \nu \in \mathcal{N}_k.
\end{align*}

For any given subdomain $D_s$, let
\begin{eqnarray}\label{def-normAi}
|\cdot|^2_{a_s} := a_s(\cdot, \cdot).
\end{eqnarray}
The following lemma can be proved.
\begin{lemma}\label{lemma-3-2}
For any given subdomain $D_s$ and vectors $\{\vec{w}_m \in
\mathbb{C}^{n_m},~m \in \mathcal{M}_s\}$, there have
\begin{eqnarray}\label{4-1-2-tildew}
\sum\limits_{m \in \mathcal{M}_s} |\bar{w}_m^{(s)}|_{a_s}^2 \le f_s |w^{(s)}|_{a_s}^2,~s = 1,\cdots,N_d,
\end{eqnarray}
where
\begin{eqnarray}\label{def-wi-wki}
w^{(s)} = \sum\limits_{m \in \mathcal{M}_s} w_m^{(s)} + \sum\limits_{m \in \mathcal{M}_s^c} w_{m,c}^{(s)},
~w_m^{(s)} = (\vec{w}_m^{(s)})^T \Phi^{m,s}, ~w_{m,c}^{(s)} \in W_{m,c}^{(s)},~\bar{w}_m^{(s)} = (\vec{w}_m^{(s)})^T \bar{\Phi}^{m,s},
\end{eqnarray}
here $\mathcal{M}_s, \mathcal{M}_s^c$ are separately defined in \eqref{def-M-i} and \eqref{def-M-i-v}, $f_s = |\mathcal{M}_s|$.
\end{lemma}
\begin{proof}
For each $m \in \mathcal{M}_s$, by using \eqref{def-wi-wki}, we obtain
\begin{eqnarray}\label{def-w-i-1}
 w^{(s)}
= w_{m}^{(s)} + \sum\limits_{\alpha \in \mathcal{M}_s \atop \alpha \neq m} w_{\alpha}^{(s)} + \sum\limits_{\alpha \in \mathcal{M}_s^c} w_{\alpha,c}^{(s)}
= \bar{w}_{m}^{(s)} + \mathbf{\eta}^{(s)},
\end{eqnarray}
where
\begin{eqnarray}\label{w-i-1-etai}
\mathbf{\eta}^{(s)} = \sum\limits_{\alpha \in \mathcal{M}_s \atop  \alpha \neq m} w_{\alpha}^{(s)} + \sum\limits_{\alpha \in \mathcal{M}_s^c} w_{\alpha,c}^{(s)} + (w_{m}^{(s)} - \bar{w}_{m}^{(s)}) \in V^{(s)} \backslash
V_{m}^{(s)}.
\end{eqnarray}

Therefore, using \eqref{def-w-i-1}, \eqref{def-wi-wki}, \eqref{w-i-1-etai}, \eqref{def-barHw-i-Fim} and Lemma \ref{lemma-ar}, we have
\begin{align*}
a_s(w^{(s)}, w^{(s)})
= a_s(\bar{w}_{m}^{(s)} + \mathbf{\eta}^{(s)},\bar{w}_{m}^{(s)} + \mathbf{\eta}^{(s)})
=a_s(\bar{w}_{m}^{(s)},\bar{w}_{m}^{(s)}) + a_s(\mathbf{\eta}^{(s)}, \mathbf{\eta}^{(s)})
\ge a_s(\bar{w}_{m}^{(s)},\bar{w}_{m}^{(s)}).
\end{align*}
From this and the definition \eqref{def-normAi} of $|\cdot|^2_{a_s}$, we can complete the proof of \eqref{4-1-2-tildew}.
\end{proof}

Then, for any given $\nu=r,j$, we introduce the scaling operator $D_{F_{k}}^{(\nu)}: U \rightarrow U (U = W_k, W_{k}^{(r)}~\mbox{or}~W_{k}^{(j)})$ or scaling matrix $\vec{D}_{F_k}^{(\nu)}  \in \mathbb{C}^{n_k\times n_k}$, which satisfy that for all $w = \vec{w}^T \Psi $ with $\vec{w} \in \mathbb{C}^{n_k}$ and $\Psi = \Phi^k, \Phi^{k,r}~\mbox{or}~\Phi^{k,j}$, we have
\begin{eqnarray}\label{def-D-operator-matrix}
D_{F_{k}}^{(\nu)} w = \vec{w}^T (\vec{D}_{F_{k}}^{(\nu)})^T \Psi,
\end{eqnarray}
where $\vec{D}_{F_{k}}^{(\nu)}$ is nonsingular, and
\begin{eqnarray}\label{def-D-trans}
D_{F_{k}}^{(r)} + D_{F_{k}}^{(j)}  = I,
\end{eqnarray}
here $I$ is the identity operator.

Two commonly used scaling matrices $\vec{D}_{F_{k}}^{(\nu)}(\nu=r,j)$ 
are the multiplicity scaling matrices
\begin{eqnarray}\label{DiFD}
\vec{D}^{(r)}_{F_k}=\frac{1}{2}\vec{I},~\vec{D}^{(j)}_{F_k}=\frac{1}{2}\vec{I},
\end{eqnarray}
and the deluxe scaling matrices (\cite{DP2012})
\begin{eqnarray}\label{DiFD-deluxe}
\vec{D}^{(r)}_{F_k}=(\vec{S}^{(r)}_{F_k} +
\vec{S}^{(j)}_{F_k})^{-1}\vec{S}^{(r)}_{F_k},~
\vec{D}^{(j)}_{F_k}=(\vec{S}^{(r)}_{F_k} +
\vec{S}^{(j)}_{F_k})^{-1}\vec{S}^{(j)}_{F_k},
\end{eqnarray}
where $\vec{I}$ denotes the $n_k \times n_k$ identity matrix, and
\begin{eqnarray}\label{Sij-def}
\vec{S}^{(\nu)}_{F_k}  = (a_{l,m}^{(\nu)})_{n_k\times n_k},~a_{l,m}^{(\nu)} = a_{\nu}(\phi^{k,\nu}_m, \phi^{k,\nu}_l),~l,m=1,\cdots, n_k,~\nu=r,j.
\end{eqnarray}

\begin{remark}
Since the spectral condition number of the plane wave discretizations of the Helmholtz equations with high wave numbers
grows like $h^{-q}$, where $q$ is proportional to the number of plane wave bases in each element $p$ (see \cite{HMP2016}),
some local techniques are introduced, for example,
by using the incomplete LU factorization preconditioner to get the deluxe scaling matrices.
\end{remark}

Using the above-mentioned scaling operators or matrices, a new set of basis functions of $W_{k}^{(\nu)} (\nu = r,j)$ can be defined as
 \begin{eqnarray}\label{ZJ-Dphik-ij-Def}
\Phi^{k,r}_D=D_{F_{k}}^{(j)} \Phi^{k,r}=(\vec{D}_{F_{k}}^{(j)})^T
\Phi^{k,r}, ~\Phi^{k,j}_D=D_{F_{k}}^{(r)}
\Phi^{k,j}=(\vec{D}_{F_{k}}^{(r)})^T \Phi^{k,j}.
\end{eqnarray}

In order to defined the so-called dual-primal basis functions,
we need to use the function spaces $W_k^{(\nu)}$ and $\bar{W}_k^{(\nu)}$ $(\nu = r,j)$ defined in \eqref{def-W-k-nu},
and the scaling operators defined in \eqref{def-D-operator-matrix}, to introduce a class of transformation operators (or matrices).

Set $n_k = n_{\Delta}^k + n_{\Pi}^k$, where the integer $n_{\Delta}^k, n_{\Pi}^k \ge 0$. Let
$n_k$-order complex nonsingular matrix
\begin{eqnarray}\label{ZJ-Oper-TFk-DelPi-Def}
\vec{T}_{F_k}=(\vec{T}_{\Delta}^{F_k},~ \vec{T}_{\Pi}^{F_k}),
\end{eqnarray}
where $\vec{T}_{\Delta}^{F_k}$ and $\vec{T}_{\Pi}^{F_k}$ are separately $n_k\times n_{\Delta}^k$ and $n_k\times n_{\Pi}^k$ matrices.

For any $\nu=r,j$, using the matrix $\vec{T}_{F_k}$, introduce the linear operators $T_{F_k}$.
These operators transform the basis vectors $\bar{\Phi}^{k,\nu}$ and $\Phi^{k,\nu}_D$ into
\begin{eqnarray}\label{ZJ-TBarphik-ij-Def}
\bar{\Phi}^{k, \nu}_T &= T_{F_k}\bar{\Phi}^{k,\nu}:=\left(
\begin{array}{l}
\bar{\Phi}^{k, \nu}_{\Delta}\\
\bar{\Phi}^{k,\nu}_{\Pi}
\end{array}\right),~
\Phi^{k, \nu}_{T_D}=T_{F_k} \Phi^{k,\nu}_D:=\left(
\begin{array}{l}
\Phi^{k, \nu}_{D,\Delta}\\
\Phi^{k,\nu}_{D,\Pi}
\end{array}\right),
\end{eqnarray}
where
\begin{eqnarray}\label{ZJ-TDphik-ij-Def}
 \bar{\Phi}^{k,\nu}_{\chi} = (\bar{\phi}^{k,\nu}_{\chi,1},\cdots,\bar{\phi}^{k,\nu}_{\chi,n_{\chi}^k})^T  = T_{\chi}^{F_k} \bar{\Phi}^{k,\nu}=
 (\vec{T}_{\chi}^{F_k})^T  \bar{\Phi}^{k,\nu},
 \Phi^{k,\nu}_{D,\chi}=T_{\chi}^{F_k} \Phi^{k,\nu}_D= (\vec{T}_{\chi}^{F_k})^T  \Phi^{k,\nu}_D, \chi=\Delta,
\Pi
\end{eqnarray}
are the so-called dual-primal basis functions.

For any given real number $\Theta \ge 1$, the above operator $T_{F_k}$ must satisfies
 \begin{eqnarray}\label{ZJ-Tphik-TDphik-Prop}
|w_{D,\Delta}^{k,r}|^2_{a_r} +
|\tilde{w}_{D,\Delta}^{k,j}|^2_{a_j} \le \Theta
|\bar{w}_{k,\Delta}^{(r)} +\bar{w}_{k,\Pi}^{(r)}|^2_{a_r},
\end{eqnarray}
where
\begin{eqnarray}\label{ZJ-Tphik-TDphik-w1}
w_{D,\Delta}^{k,r} =(\vec{w}_\Delta)^T
\Phi^{k,r}_{D,\Delta},~\tilde{w}_{D,\Delta}^{k,j} =(\vec{w}_\Delta)^T
\Phi^{k,j}_{D,\Delta},~\bar{w}_{k,\chi}^{(r)}=
(\vec{w}_\chi)^T  \bar{\Phi}^{k,r}_{\chi},~\chi=\Delta, \Pi,
\end{eqnarray}
here $\vec{w}_{\Delta} \in \mathbb{C}^{n_{\Delta}^k}$, $\vec{w}_{\Pi} \in \mathbb{C}^{n_{\Pi}^k}$ are any given vectors.

The inequality \eqref{ZJ-Tphik-TDphik-Prop} plays a crucial role in the estimation of the condition number of the adaptive BDDC algorithm, and it is always
be replaced by lazy eigenanalysis.  
Following \cite{KCX2017,KCW2017,PSW2017}, we introduce the matrices
\begin{align*}
\vec{\bar{S}}^{(\nu)}_{F_k}  = (b_{l,m}^{(\nu)})_{n_k\times n_k},~b_{l,m}^{(\nu)} = a_{\nu}(\bar{\phi}^{k,\nu}_m, \bar{\phi}^{k,\nu}_l),~l,m=1,\cdots, n_k,~\nu=r,j,
\end{align*}
where the sesquilinear form $a_{\nu}(\cdot,\cdot)$ and the basis functions $\{\bar{\phi}^{k,\nu}_l\}_{l=1}^{n_k}$ are seperately defined in \eqref{def-AUV} and \eqref{def-Phi-Psi-k-nu}.

Then considering a generalized eigenvalue problem (see \cite{PD2013,KCW2017,PC2016,KCX2017})
\begin{eqnarray}\label{eig-pro-intro}
\vec{A}_{F_k}^D \vec{v}
 = \lambda \vec{B}_{F_k} \vec{v},
\end{eqnarray}
where
\begin{eqnarray}\label{Def-A-B}
\vec{A}_{F_k}^D = (\vec{D}_{F_k}^{(r)})^H \vec{S}_{F_k}^{(j)} \vec{D}_{F_k}^{(r)} + (\vec{D}_{F_k}^{(j)})^H \vec{S}_{F_k}^{(r)} \vec{D}_{F_k}^{(j)},
~~\vec{B}_{F_k} = \vec{\bar{S}}_{F_{k}}^{(r)}:\vec{\bar{S}}_{F_{k}}^{(j)},
\end{eqnarray}
here $\diamond^H$ denotes the conjugate transpose of $\diamond$,  $\vec{v} \in \mathbb{C}^{n_k}$, $\lambda \in \mathbb{C}$,
$\vec{D}_{F_k}^{(\nu)}(\nu=r,j)$, $\vec{S}^{(\nu)}_{F_k}(\nu=r,j)$ are separately defined in \eqref{def-D-operator-matrix} and \eqref{Sij-def},
and the parallel sum
$$\vec{B}_{F_k} =
\vec{\bar{S}}_{F_{k}}^{(j)}(\vec{\bar{S}}_{F_{k}}^{(r)}+\vec{\bar{S}}_{F_{k}}^{(j)})^{\dagger}\vec{\bar{S}}_{F_{k}}^{(r)},$$
here
$(\vec{\bar{S}}_{F_{k}}^{(r)}+\vec{\bar{S}}_{F_{k}}^{(j)})^{\dagger}$
is a pseudo inverse of the matrix
$\vec{\bar{S}}_{F_{k}}^{(r)}+\vec{\bar{S}}_{F_{k}}^{(j)}$.

Since $\vec{\bar{S}}_{F_{k}}^{(\nu)}(\nu=r,j)$ are both Hermitian positive semi-definite,  $\vec{\bar{S}}_{F_{k}}^{(r)}:\vec{\bar{S}}_{F_{k}}^{(j)}$
is also Hermitian positive semi-definite and satisfies the following spectral inequalities \cite{1969AD}
\begin{eqnarray}\label{parallel-property}
\vec{B}_{F_k} \le \vec{\bar{S}}_{F_{k}}^{(\nu)},~\nu=r,j.
\end{eqnarray}

Let
\begin{eqnarray}\label{lambda-Theta}
|\lambda_1| \le |\lambda_2| \le \cdots \le |\lambda_{n_{\Delta}^k}| \le \Theta \le |\lambda_{n_{\Delta}^k+1}| \le \cdots \le |\lambda_{n_k}|,
\end{eqnarray}
where $\lambda_k (k=1,\cdots,n_k)$ is the eigenvalue of \eqref{eig-pro-intro}, $n_{\Delta}^k$ is a non-negative integer, and $\Theta \ge 1$ is given in \eqref{ZJ-Tphik-TDphik-Prop}.

Denote $\vec{T}_{\Delta}^{F_k}$ and $\vec{T}_{\Pi}^{F_k}$ in the $n_k \times n_k$ transform matrix $\vec{T}_{F_k}$ defined in \eqref{ZJ-Oper-TFk-DelPi-Def} as
\begin{align*}
\vec{T}_{\Delta}^{F_k}:= (\vec{v}_1,\cdots,\vec{v}_{n_{\Delta}^k}),~~\vec{T}_{\Pi}^{F_k}:= (\vec{v}_{n_{\Delta}^k+1},\cdots,\vec{v}_{n_k}),
\end{align*}
here $\vec{v}_l(l=1,\cdots,n_k)$ are the generalized eigenvectors  of \eqref{eig-pro-intro} corresponding to $\lambda_l$ and
satisfy
\begin{eqnarray}\label{oth-1}
\vec{v}_l^H \vec{A}_{F_k}^D \vec{v}_m = \vec{v}_l^H \vec{B}_{F_k} \vec{v}_m = 0,~\mbox{if}~l \neq m.
\end{eqnarray}



From this, we have
\begin{eqnarray}\label{OrthProp-Tk-Delta}
(\vec{T}_{\Delta}^{F_k})^H \vec{C} \vec{T}_{\Delta}^{F_k}= diag
(\vec{v}_1^H \vec{C} \vec{v}_1,\cdots, \vec{v}_{n_{\Delta}^k}^H \vec{C}
\vec{v}_{n_{\Delta}^k}),~~\vec{C} = \vec{A}_{F_k}^D, \vec{B}_{F_k},
\end{eqnarray}
and
\begin{eqnarray}\label{OrthProp-Tk-Delta-Pi}
(\vec{T}_{\Pi}^{F_k})^H \vec{B}_{F_k}
\vec{T}_{\Delta}^{F_k}=0,~ (\vec{T}_{\Delta}^{F_k})^H \vec{B}_{F_k}
\vec{T}_{\Pi}^{F_k} = 0.
\end{eqnarray}
Using \eqref{eig-pro-intro}, we can prove that
\begin{eqnarray}\label{remark-equ-general-eig-pro}
\vec{v}_l^H \vec{A}_{F_k}^D \vec{v}_l = |\lambda_l| \vec{v}_l^H \vec{B}_{F_k} \vec{v}_l,~l=1,\cdots,n_{\Delta}^k.
\end{eqnarray}

%
%
%
%

 Using the above matrix $\vec{T}_{F_k}$, we can obtain the operator
$T_{F_k}$ defined in \eqref{ZJ-TBarphik-ij-Def}. Next, we verify that it satisfies \eqref{ZJ-Tphik-TDphik-Prop}.

By \eqref{ZJ-TDphik-ij-Def} and \eqref{ZJ-Dphik-ij-Def}, we can rewrite the functions in \eqref{ZJ-Tphik-TDphik-w1} as
\begin{eqnarray}\label{Esp-wk-Del-Td-nu}
w_{D,\Delta}^{k,r} =(\vec{w}_\Delta)^T
(\vec{T}_{\Delta}^{F_k})^T  (\vec{D}_{F_{k}}^{(j)})^T \Phi^{k,r},~
\tilde{w}_{D,\Delta}^{k,j} =(\vec{w}_\Delta)^T
(\vec{T}_{\Delta}^{F_k})^T  (\vec{D}_{F_{k}}^{(r)})^T \Phi^{k,j},
\end{eqnarray}
and
\begin{eqnarray}\label{Esp-Barwk-chi-Ti}
\bar{w}_{\chi}^{k,r}=
(\vec{w}_\chi)^T (\vec{T}_{\chi}^{F_k})^T  \bar{\Phi}^{k,\nu},~\chi=\Delta, \Pi.
\end{eqnarray}

From \eqref{def-normAi}, \eqref{Esp-wk-Del-Td-nu}, the property of the sesquilinear form $a_r(\cdot,\cdot)$ and \eqref{Sij-def}, it is easy to verify that
\begin{eqnarray}\label{Ai-Dkj-wkDeltai}
|w_{D,\Delta}^{k,r}|^2_{a_r} = a_r(w_{D,\Delta}^{k,r}, w_{D,\Delta}^{k,r})
= \vec{w}_{\Delta}^H(\vec{T}_{\Delta}^{F_k})^H (\vec{D}_{F_k}^{(j)})^H\vec{S}_{F_k}^{(r)}
 \vec{D}_{F_k}^{(j)} \vec{T}_{\Delta}^{F_k} \vec{w}_{\Delta}.
\end{eqnarray}
Similarly, we have
 \begin{eqnarray}\label{Aj-Dki-wkDeltaj}
|\tilde{w}_{D,\Delta}^{k,j}|^2_{a_j} = \vec{w}_{\Delta}^H (\vec{T}_{\Delta}^{F_k})^H (\vec{D}_{F_k}^{(r)})^H \vec{S}_{F_k}^{(j)}
 \vec{D}_{F_k}^{(r)} \vec{T}_{\Delta}^{F_k} \vec{w}_{\Delta}.
\end{eqnarray}

By using \eqref{Ai-Dkj-wkDeltai}, \eqref{Aj-Dki-wkDeltaj}, \eqref{Def-A-B} and \eqref{OrthProp-Tk-Delta},
we can obtain
\begin{eqnarray*} 
|w_{D,\Delta}^{k,r}|^2_{a_r} +
|\tilde{w}_{D,\Delta}^{k,j}|^2_{a_j} = \vec{w}_{\Delta}^H
 (\vec{T}_{\Delta}^{F_k})^H \vec{A}_{F_k}^D  \vec{T}_{\Delta}^{F_k}
 \vec{w}_{\Delta}
= \vec{w}_{\Delta}^H diag\{\vec{v}_1^H \vec{A}_{F_k}^D \vec{v}_1,\cdots, \vec{v}_{n_{\Delta}^k}^H \vec{A}_{F_k}^D \vec{v}_{n_{\Delta}^k}\}\vec{w}_{\Delta}.
\end{eqnarray*}

From this, and using \eqref{remark-equ-general-eig-pro}, \eqref{lambda-Theta}, \eqref{OrthProp-Tk-Delta}, \eqref{OrthProp-Tk-Delta-Pi}, \eqref{parallel-property}, \eqref{Esp-Barwk-chi-Ti}, and note that $\vec{B}_{F_k}$ is Hermitian positive semi-definite, we known that
\begin{align*}
|w_{D,\Delta}^{k,r}|^2_{a_r} +
|\tilde{w}_{D,\Delta}^{k,j}|^2_{a_j}
&= \vec{w}_{\Delta}^H diag\{|\lambda_1| \vec{v}_1^H \vec{B}_{F_k} \vec{v}_1,\cdots, |\lambda_{n_{\Delta}^k}| \vec{v}_{n_{\Delta}^k}^H \vec{B}_{F_k} \vec{v}_{n_{\Delta}^k}\}\vec{w}_{\Delta}
\\\nonumber
& \le \Theta \vec{w}_{\Delta}^H diag\{\vec{v}_1^H \vec{B}_{F_k} \vec{v}_1,\cdots, \vec{v}_{n_{\Delta}^k}^H \vec{B}_{F_k} \vec{v}_{n_{\Delta}^k}\}\vec{w}_{\Delta}\\
&= \Theta \vec{w}_{\Delta}^H (\vec{T}_{\Delta}^{F_k})^H \vec{B}_{F_k} \vec{T}_{\Delta}^{F_k} \vec{w}_{\Delta}
\\
&\le \Theta (\vec{T}_{\Delta}^{F_{k}}\vec{w}_{\Delta} + \vec{T}_{\Pi}^{F_{k}}\vec{w}_{\Pi})^H \vec{B}_{F_k} (\vec{T}_{\Delta}^{F_{k}}\vec{w}_{\Delta} + \vec{T}_{\Pi}^{F_{k}}\vec{w}_{\Pi})
\\\nonumber
&\le \Theta (\vec{T}_{\Delta}^{F_{k}}\vec{w}_{\Delta} + \vec{T}_{\Pi}^{F_{k}}\vec{w}_{\Pi})^H \vec{\bar{S}}_{F_{k}}^{(r)} (\vec{T}_{\Delta}^{F_{k}}\vec{w}_{\Delta} + \vec{T}_{\Pi}^{F_{k}}\vec{w}_{\Pi})
\\\nonumber
&= \Theta
a_r(\bar{w}_{k,\Delta}^{(r)} + \bar{w}_{k,\Pi}^{(r)}, \bar{w}_{k,\Delta}^{(r)} + \bar{w}_{k,\Pi}^{(r)})
\\
&= \Theta |\bar{w}_{k,\Delta}^{(r)} + \bar{w}_{k,\Pi}^{(r)}|_{a_r}^2
\end{align*}
Then \eqref{ZJ-Tphik-TDphik-Prop} holds. $\Box$

Further, utilizing the linear operator $T_{F_k}$ (or matrix $\vec{T}_{F_k}$) defined above, we can transform the basis function vector $\Phi^k$ of $W_k$ into the so-called new dual-primal basis function vector
\begin{eqnarray*}
 \Phi_T^{k} = T_{F_k} \Phi^{k} :=
\left(\begin{array}{l}
\Phi^{k}_{\Delta}\\
\Phi^{k}_{\Pi}
\end{array}\right),
\end{eqnarray*}
where
\begin{eqnarray}\label{ZJ-Def-Bases-WkDel-WkPi}
 \Phi^{k}_{\chi} = (\phi^{k}_{\chi,1},\cdots,\phi^{k}_{\chi,n_{\chi}^k})^T = T_{\chi}^{F_k} \Phi^{k}:=
 (\vec{T}_{\chi}^{F_k})^T \Phi^{k},~\chi=\Delta,
\Pi.
\end{eqnarray}

From this, we can decompose $W_{k}$ into
\begin{eqnarray}\label{ZJ-Def-Wk}
W_{k} = W_{k,\Delta}\oplus  W_{k,\Pi},
\end{eqnarray}
where the function spaces $W_{k,\Delta}$ and $W_{k,\Pi}$ ($k=1,\cdots, N_f$) are separately formed by the component functions of $\Phi^{k}_{\Delta}$ and $\Phi^{k}_{\Pi}$.

Then, using \eqref{ZJ-Def-Wk}, a decomposition of the space $\hat{W}$ can be obtained as follows
\begin{eqnarray}\label{Zj-Decomp-HatW}
\hat{W} = W_{\Delta} \oplus W_{\Pi},
\end{eqnarray}
where the dual and primal variable space 
\begin{eqnarray}\label{Zj-Def-WDelPi}
W_{\Delta}=\oplus_{k=1}^{N_f} W_{k,\Delta} ,~W_{\Pi} = (\oplus_{k=1}^{N_f} W_{k,\Pi}) \oplus (\oplus_{k=1}^{N_v} W_{k,c}).
\end{eqnarray}
%


Similarly, using the linear operator $T_{F_k}$ (or matrix $\vec{T}_{F_k}$), we can transform the basis function vector $\Phi^{k,\nu}$ of
$W_{k}^{(\nu)}$ into the dual-primal basis function vector
\begin{eqnarray*}
 \Phi_T^{k,\nu} = T_{F_k} \Phi^{k,\nu} :=
\left(\begin{array}{l}
\Phi^{k,\nu}_{\Delta}\\
\Phi^{k,\nu}_{\Pi}
\end{array}\right),
\end{eqnarray*}
where
\begin{eqnarray}\label{Zj-Def-Vec-TildPhiDel-1}
 \Phi^{k,\nu}_{\chi} = (\phi^{k,\nu}_{\chi,1},\cdots,\phi^{k,\nu}_{\chi,n_{\chi}^k})^T = T_{\chi}^{F_k} \Phi^{k,\nu}:=
 (\vec{T}_{\chi}^{F_k})^T \Phi^{k,\nu},~\chi=\Delta,
\Pi.
\end{eqnarray}

By using the definitions \eqref{ZJ-Def-Bases-WkDel-WkPi},
\eqref{Zj-Def-Vec-TildPhiDel-1} of $\{\phi^{k}_{\chi,l}\}$ and $\{\phi^{k,\nu}_{\chi,l}\}$ $(\chi = \Delta,\Pi, \nu=i,j)$, and \eqref{rel-phik-phikv}, we get
\begin{align*}
\phi^{k}_{\chi,l}|_{\bar{D}_{\nu}} = \phi^{k,\nu}_{\chi,l}|_{\bar{D}_{\nu}},~l=1,\cdots,n_{\chi}^k, \chi = \Delta,\Pi.
\end{align*}

Decompose $W_{k}^{(\nu)}$ into
\begin{eqnarray}\label{Zj-Def-WkDelnu-WkPinu}
W_{k}^{(\nu)} = W_{k,\Delta}^{(\nu)} \oplus
W_{k,\Pi}^{(\nu)},~ \nu=i,j,
\end{eqnarray}
where the basis function vector of $W_{k,\Delta}^{(\nu)}$ and $W_{k,\Pi}^{(\nu)}$ are $\Phi^{k,\nu}_{\Delta}$ and $\Phi^{k,\nu}_{\Pi}$, respectively.

Using the above decomposition, the partially coupled function space which is relied on the adaptive BDDC preconditioner can be obtained and expressed as
\begin{eqnarray}\label{ZJ-TildeW-Decomp}
\tilde{W} = \tilde{W}_{\Delta} \oplus W_{\Pi},
\end{eqnarray}
where
$W_{\Pi}$ is defined in \eqref{Zj-Def-WDelPi}, and
\begin{eqnarray}\label{ZJ-TildeW-Delta-Decomp}
\tilde{W}_{\Delta} = \oplus_{r=1}^{N_d} W_{\Delta}^{(r)},
~W_{\Delta}^{(r)} =\oplus_{k\in \mathcal{M}_r}
W_{k,\Delta}^{(r)} = \oplus_{l=1}^{f_r}
W_{r_l,\Delta}^{(r)},~r=1,\cdots, N_d,
\end{eqnarray}
here $\mathcal{M}_r$ is defined in \eqref{def-M-i}, $f_r$ denotes the size of $\mathcal{M}_r$, and the subspace $W_{r_l, \Delta}^{(r)}$ is defined in \eqref{Zj-Def-WkDelnu-WkPinu}.



Further, for the need of the theoretical analysis, we can present the decomposition of the auxiliary space $\bar{W}_{k}^{(\nu)}$ based on the dual-primal basis function vector as
\begin{eqnarray*}
\bar{W}_{k}^{(\nu)} = \bar{W}_{k,\Delta}^{(\nu)} \oplus
\bar{W}_{k,\Pi}^{(\nu)},
\end{eqnarray*}
where the basis function vectors $\bar{\Phi}_{\Delta}^{k,\nu}$ and $\bar{\Phi}_{\Pi}^{k,\nu}$ of the subspaces $\bar{W}_{k,\Delta}^{(\nu)}$ and
$\bar{W}_{k,\Pi}^{(\nu)}$ are defined in \eqref{ZJ-TDphik-ij-Def}, respectively.

In the following, the dual-primal basis function vectors $\{\Phi_{\Delta}^k\}$ and $\{\Phi_{\Pi}^k\}$ for the Schur complement space $\hat{W}$ will be adopted.
And we will design and analyze the adaptive BDDC preconditioner for the corresponding Schur complement system  \eqref{def-schur-system}.
%
%
%
\subsection{BDDC preconditioner}\label{sec:4}

We focus on the two-level adaptive BDDC preconditioner firstly.
In order to describe this preconditioner, we need to introduce some commonly used linear operators firstly.

Let $R_U^V$ be the restriction operator from the Hilbert space $U$ onto its subspaces $V$,
and $I_V^U$ be the interpolation operator from $V$ to $U$ (see reference \cite{PSW2017}).
In particular, when $V=U$, $I_{U}^{U}$ (or $I_{V}^{V}$) is an identity operator.

Denote the Hilbert spaces $Z = span\{\phi_1^Z,\cdots,\phi_n^Z\}$ and $W = span\{\phi_1^W,\cdots,\phi_n^W\}$, define the linear basis transformation operator $T_Z^W: Z\rightarrow W$ such that
\begin{eqnarray}\label{ZJ-Def-Tran-ZtoW}
 T_Z^W \phi^{Z}_l =\phi^{W}_l,~l=1,\cdots, n.
\end{eqnarray}
In particular, for any $k=1, \cdots, N_f$ and $\nu=r, j$, we have
\begin{eqnarray}\label{ZJ-Basis-WkiToWk-Tran}
&&T_{W_k}^{W_k^{(\nu)}} \phi^{k}_{\chi, l}= \phi^{k,
\nu}_{\chi,l},~
T^{W_k}_{W_k^{(\nu)}} \phi^{k, \nu}_{\chi, l} =
\phi^{k}_{\chi, l},~
l=1,\cdots, n_{\chi}^k,~\chi = \Delta,\Pi,\\\label{ZJ-Basis-WkiToWk-Tran-ij}
&& T_{W_k^{(r)}}^{W_k^{(j)}} \phi^{k,r}_{\chi,l}= \phi^{k,j}_{\chi,l},~T_{W_k^{(\nu)}}^{\bar{W}_k^{(\nu)}} \phi^{k,\nu}_{\chi,l} = \bar{\phi}^{k,\nu}_{\chi,l},~
l=1,\cdots, n_{\chi}^k,~\chi = \Delta,\Pi,
\end{eqnarray}
and for any $k =1,\cdots,N_v$ and $r \in \mathcal{N}_k$, we have
\begin{eqnarray}\label{T-Wkc-Wkcr}
T_{W_{k,c}}^{W_{k,c}^{(r)}} \psi^{k}_{l}= \psi^{k,r}_{l},~l=1,\cdots,n_k^c.
\end{eqnarray}

For a given linear operator $L$ from the Hilbert space $U$ to the Hilbert space $V$, the complex conjugate transpose
operator $L^H: V \rightarrow U$ is defined by
\begin{align*}
(L^H v, u) = (v, Lu),~~\forall u \in U, v\in V.
\end{align*}

Then, using the basis transformation operators $T_{W_k}^{W_k^{(r)}}$ and $T_{W_{k,c}}^{W_{k,c}^{(r)}}$, another sesquilinear form on $\tilde{W}$ can be introduced.

For any $\tilde{u},\tilde{v} \in \tilde{W}$, using the decomposition \eqref{ZJ-TildeW-Decomp}, \eqref{ZJ-TildeW-Delta-Decomp} and \eqref{Zj-Def-WDelPi} of $\tilde{W}$, we have
\begin{eqnarray}\label{ZJ-TildeW-FuncDecomp}
\tilde{\zeta} = \sum\limits_{r=1}^{N_d} \sum\limits_{k\in
\mathcal{M}_r} \tilde{\zeta}_{k,\Delta}^{(r)} + \sum\limits_{k=1}^{N_f}
\tilde{\zeta}_{k,\Pi} + \sum\limits_{k=1}^{N_v}
\tilde{\zeta}_{k,c},~~\tilde{\zeta} = \tilde{u}, \tilde{v},
\end{eqnarray}
where $\tilde{\zeta}_{k,\Delta}^{(r)}\in W_{k,\Delta}^{(r)}$,
$\tilde{\zeta}_{k,\Pi}\in W_{k,\Pi}$ and $\tilde{\zeta}_{k,c} \in W_{k,c}$.
From this, we can define a sesquilinear form
$\tilde{a}(\cdot,\cdot)$ and its corresponding HPD operator $\tilde{S}: \tilde{W} \rightarrow \tilde{W}$ such that
\begin{eqnarray}\label{def-operator-tilde-A}
(\tilde{S} \tilde{u}, \tilde{v}) :=\tilde{a}(\tilde{u},
\tilde{v}) = \sum\limits_{r=1}^{N_d} a_r(\tilde{u}^{(r)},
\tilde{v}^{(r)}),~~\forall \tilde{u}, \tilde{v} \in
\tilde{W},
\end{eqnarray}
where
\begin{eqnarray}\label{def-operator-tilde-S}
\tilde{\zeta}^{(r)} = \sum\limits_{k \in \mathcal{M}_r}
(\tilde{\zeta}_{k,\Delta}^{(r)} + T_{W_k}^{W_k^{(r)}}
\tilde{\zeta}_{k,\Pi}) + \sum\limits_{k\in \mathcal{M}_r^c}T_{W_{k,c}}^{W_{k,c}^{(r)}}
\tilde{\zeta}_{k,c},~\tilde{\zeta} = \tilde{u}, \tilde{v}.
\end{eqnarray}


For any $r=1,\cdots, N_d$, we define a linear operator $\hat{I}_{W_{\Delta}^{(r)}}^{D}: W_{\Delta}^{(r)} \rightarrow
\hat{W}$ such that
\begin{eqnarray}\label{ZJ-Def-Oper-HatI-DelDi}
\hat{I}_{W_{\Delta}^{(r)}}^{D}
=\sum\limits_{k \in \mathcal{M}_r} T^{W_k}_{W_k^{(r)}}
D_{F_k}^{(r)} R_{W_{\Delta}^{(r)}}^{W_{k,\Delta}^{(r)}},
\end{eqnarray}
where $R_{W_{\Delta}^{(r)}}^{W_{k,\Delta}^{(r)}}$ is a restriction operator from $W_{\Delta}^{(r)}$ to its subspace $W_{k,\Delta}^{(r)}$.
It's easy to verify that for any given $k \in \mathcal{M}_r$, we can obtain
\begin{eqnarray}\label{def-R-property-1}
\hat{I}_{W_{\Delta}^{(r)}}^{D} w =  T^{W_k}_{W_k^{(r)}}
D_{F_k}^{(r)} w,~\forall w  \in W_{k,\Delta}^{(r)}.
\end{eqnarray}

Using $\hat{I}_{W_{\Delta}^{(r)}}^{D} (r=1,\cdots,N_d)$ and the restriction operators $R_{\tilde{W}}^{W_{\Delta}^{(r)}} (r=1,\cdots,N_d)$ and $R_{\tilde{W}}^{W_{\Pi}}$, we can introduce an average operator $E_D: \tilde{W} \rightarrow \hat{W}$, which satisfy
\begin{eqnarray}\label{ZJ-Def-ED}
E_D = Q_D+ R_{\tilde{W}}^{W_{\Pi}},
\end{eqnarray}
where
\begin{eqnarray}\label{ZJ-Def-QD}
Q_D=\sum\limits_{r=1}^{N_d} \hat{I}_{W_{\Delta}^{(r)}}^{D}
R_{\tilde{W}}^{W_{\Delta}^{(r)}}.
\end{eqnarray}

With the above-mentioned preparations, by using the sesquilinear form $\tilde{a}(\cdot,\cdot)$ and the average operator $E_D$,
the adaptive BDDC operator $M_{BDDC}^{-1}: \hat{W} \rightarrow \hat{W}$ for solving the Schur
complement system \eqref{def-schur-system} can then be given in the following algorithm.
\begin{algorithm}\label{chap2-ZJ-algorithm-variation-BDDC-domain-simple}~

For any given ${g}\in \hat{W}$, $u_g = M_{BDDC}^{-1} {g} \in \hat{W}$ can be obtained by the following two steps.
\begin{description}
\item[Step 1.] Compute $w \in \tilde{W}$ by
\begin{align*}
\tilde{a}(w, v) = ((E_D)^H g, v),~~\forall v \in \tilde{W}.
\end{align*}

\item[Step 2.] Let
\begin{align*}
u_g = E_D w.
\end{align*}

\end{description}
\end{algorithm}

From this algorithm, using the definition \eqref{def-operator-tilde-A} of $\tilde{S}$ and note that $\tilde{S}$ is Hermitian positive definite, it is easy to verify that $M_{BDDC}^{-1}$ can be written as
\begin{eqnarray}\label{ZJ-Expression-MBDDC-Inv}
M_{BDDC}^{-1} = E_D \tilde{S}^{-1} (E_D)^H.
\end{eqnarray}

In order to facilitate parallel programming, we can give an equivalent description of Algorithm \ref{chap2-ZJ-algorithm-variation-BDDC-domain-simple}.
For this purpose, some operators are introduced firstly.

Using the prolongation operators $I_{W_{\Delta}^{(r)}}^{\tilde{W}}$ ($r=1,\cdots,N_d$), a linear operator from $\tilde{W}$ to $\tilde{W}_{\Delta}$ can be defined as
\begin{eqnarray}\label{ZJn-Def-TildeSDel-Inv}
\tilde{S}_{\Delta}^{-1}= \sum\limits_{r=1}^{N_d}
(\tilde{S}^{(r)}_{\Delta\Delta})^{-1}(I_{W_{\Delta}^{(r)}}^{\tilde{W}})^H = \sum\limits_{r=1}^{N_d}
I_{W_{\Delta}^{(r)}}^{\tilde{W}}
(\tilde{S}^{(r)}_{\Delta\Delta})^{-1}(I_{W_{\Delta}^{(r)}}^{\tilde{W}})^H,
\end{eqnarray}
where
\begin{eqnarray}\label{ZJ-Oper-TilSi-Del-Pi}
\tilde{S}^{(r)}_{\Delta\Delta} =
(I_{W_{\Delta}^{(r)}}^{\tilde{W}})^H \tilde{S}
I_{W_{\Delta}^{(r)}}^{\tilde{W}}.
\end{eqnarray}

Further, utilizing $\tilde{S}_{\Delta}^{-1}$, we can introduce linear operator $O_{\tilde{\Pi}}: W_{\Pi} \rightarrow \tilde{W}$ as
\begin{eqnarray}\label{ZJ-Repres-OPi-Eqiv}
O_{\tilde{\Pi}}= I_{W_{\Pi}}^{\tilde{W}} - \tilde{S}_{\Delta}^{-1} \tilde{S}
I_{W_{\Pi}}^{\tilde{W}} = (I_{\tilde{W}}^{\tilde{W}} - \tilde{S}_{\Delta}^{-1} \tilde{S}) I_{W_{\Pi}}^{\tilde{W}},
\end{eqnarray}
where $I_{W_{\Pi}}^{\tilde{W}}$ is the prolongation operator from $W_{\Pi}$ to $\tilde{W}$, $I_{\tilde{W}}^{\tilde{W}}$ is an identity operator on $\tilde{W}$.

Thus, by using the expression \eqref{ZJ-Expression-MBDDC-Inv} of the adaptive BDDC operator $M_{BDDC}^{-1}$,
refer to the derivation process of Theorem 4.1 in \cite{PSW2017}, we can see that Algorithm \ref{chap2-ZJ-algorithm-variation-BDDC-domain-simple} can be described as equivalent as follows.
\begin{algorithm}\label{ZJ-algorithm-variation-BDDC-domain}~

For any given ${g}\in \hat{W}$, $u_g = M_{BDDC}^{-1} {g}
\in \hat{W}$ can be obtained from the four steps.
\begin{description}
\item[Step 1.] Compute
$u^{\Delta,r}_a \in W_{\Delta}^{(r)}(r=1,\cdots,N_d)$ in parallel by
\begin{align*}
a_r(u^{\Delta,r}_a, v) = ((\hat{I}_{W_{\Delta}^{(r)}}^{D})^H {g}, v),~~\forall v \in W_{\Delta}^{(r)},
\end{align*}
and set
\begin{align*}
u_{\Delta,a} = \sum\limits_{r=1}^{N_d} \hat{I}_{W_{\Delta}^{(r)}}^{D} u^{\Delta,r}_a \in \hat{W},
\end{align*}
where the operator $\hat{I}_{W_{\Delta}^{(r)}}^{D}$ is defined in \eqref{ZJ-Def-Oper-HatI-DelDi}.


\item[Step 2.]  Compute $u_{\Pi} \in W_{\Pi}$ by
\begin{align*}
\tilde{a}(O_{\tilde{\Pi}} u_{\Pi}, O_{\tilde{\Pi}}v) = (g, v) - \tilde{a}(\sum\limits_{r=1}^{N_d} u^{\Delta,r}_a,v),~~\forall v\in
W_{\Pi},
\end{align*}
where the operator $O_{\tilde{\Pi}}$ is defined in \eqref{ZJ-Repres-OPi-Eqiv}.

\item[Step 3.] Compute $u^{\Delta,r}_{b} \in W_{\Delta}^{(r)}(r=1,\cdots,N_d)$ in parallel by
\begin{align*}
a_r(u^{\Delta,r}_{b}, v) = - a_r(u_{\Pi}, v),~~\forall v \in W_{\Delta}^{(r)},
\end{align*}
and set
\begin{align*}
u_{\Delta,b} = \sum\limits_{r=1}^{N_d} \hat{I}_{W_{\Delta}^{(r)}}^{D} u^{\Delta,r}_b \in \hat{W}.
\end{align*}


\item[Step 4.] Set
\begin{align*}
u_g = u_{\Delta,a} + u_{\Pi} + u_{\Delta,b}.
\end{align*}

\end{description}

\end{algorithm}

Since Algorithm \ref{ZJ-algorithm-variation-BDDC-domain} can be considered as a two-level algorithm, we also call Algorithm \ref{ZJ-algorithm-variation-BDDC-domain} or Algorithm \ref{chap2-ZJ-algorithm-variation-BDDC-domain-simple} two-level adaptive BDDC algorithm.

Further, by using Algorithm \ref{chap2-ZJ-algorithm-variation-BDDC-domain-simple} or Algorithm \ref{ZJ-algorithm-variation-BDDC-domain},
we can get the following algorithm for solving the original variational problem \eqref{115-dis}.
\begin{algorithm}\label{chap2-ZJ-algorithm-variation-BDDC-domain-II}~

\begin{description}
\item[Step 1.]  Using the Krylov subspace iteration methods based on $M_{BDDC}^{-1}$ preconditioner to find $u_{\Gamma} \in \hat{W}$ such that
\begin{align*}
a(u_{\Gamma}, v) = \mathcal{L}(v),~\forall v \in \hat{W}.
\end{align*}

\item[Step 2.]  Compute $u^{(r)}_I \in V_I^{(r)}(r=1,\cdots,N_d)$ in parallel by
\begin{align*}
a_r(u^{(r)}_I, v) = \mathcal{L}(v) - a_r(u_{\Gamma}, v),~~\forall v \in V_I^{(r)}.
\end{align*}

\item[Step 3.] Set
\begin{align*}
u = \sum\limits_{r=1}^{N_d} u^{(r)}_I + u_{\Gamma}.
\end{align*}
\end{description}
\end{algorithm}

From the results of numerical experiments for the PWLS discretizations of the Helmholtz equations, we find that
the number of primal unknowns increases as the number of wave numbers and subdomains increase, and the corresponding coarse problem will become too large and hard to solve directly. This leads to multilevel extension of this algorithm naturally. Multilevel BDDC algorithm were first presented in \cite{Dohrmann2003}, and further developed in \cite{T20071,T20072,MSD2008,ZT2017}. In addition, the multilevel preconditioners for solving the systems arising from the plane wave discretizations for Helmholtz equations with large wave numbers were constructed by Hu and Li in \cite{HL20171,HL20172}.

Following \cite{T20071,T20072,MSD2008,ZT2017}, in our multilevel algorithm, we denote the $s$th level mesh by $\mathcal{T}_h^s$ $(s=0,\cdots,L-1)$,
where $L$ is the total number of levels, $\mathcal{T}_h^0 = \mathcal{T}_h$ is the finest level,
and a subdomain at a finer level is considered as an element of a coarser mesh.
Let $V_p^s({\mathcal T_h^s})$ be the finite element spaces of the original problem associated with $\mathcal{T}_h^s$,
and set $V_p^{s+1}({\mathcal T_h^{s+1}}) := W_{\Pi}^s$,
where $W_{\Pi}^s$ is the coarse space at level $s$.
In particular, $V_p^0({\mathcal T_h^0}) = V_p(\mathcal{T}_h)$ and $V_p^1({\mathcal T_h^1}) = W_{\Pi}$ are defined in \eqref{Vp-space} and \eqref{Zj-Def-WDelPi}, respectively. Noticing that Algorithm \ref{chap2-ZJ-algorithm-variation-BDDC-domain-II} gives an iteration process from level $s$ to level $s+1$ ($s=0$), and by using the solution of the Schur complete problem, the solution of the original problem at level $0$ can be obtained.
Iterating this procedure until $s < L-1$, and we can arrive at our multilevel adaptive BDDC algorithm.
A relation between the condition number of the multilevel BDDC algorithm with corner coarse function
in 2D and each level problem was presented in \cite{MSD2008}.

In the next section, we will provide the condition number estimate of the two-level adaptive BDDC preconditioned operator.

\section{Analysis of condition number bound}\label{sec:5}
\setcounter{equation}{0}

We first establish the relation between the operators $\tilde{S}$ and $\hat{S}$, which are defined in \eqref{def-operator-tilde-A} and \eqref{def-operator-hat-S}, respectively. For this purpose, we introduce a subspace of $\tilde{W}$ such that
 \begin{eqnarray*}
\bar{\tilde{W}} =\bar{\tilde{W}}_{\Delta} \oplus W_{\Pi},
\end{eqnarray*}
where $W_{\Pi}$ is defined in \eqref{Zj-Def-WDelPi},  and
\begin{align*}
\bar{\tilde{W}}_{\Delta} = \oplus_{k=1}^{N_f} \bar{\tilde{W}}_{k,\Delta},~~\bar{\tilde{W}}_{k,\Delta} = span\{\phi^{k,r}_{\Delta,1}
+\phi^{k,j}_{\Delta,1},\cdots,\phi^{k,r}_{\Delta,n_{\Delta}^k}
+\phi^{k,j}_{\Delta,n_{\Delta}^k}\},
\end{align*}
here $\{\phi^{k, \nu}_{\Delta, l}\}$ is defined in \eqref{Zj-Def-Vec-TildPhiDel-1}.

Using \eqref{ZJ-Def-Tran-ZtoW}, we can define another linear basis transformation operator $T_{\hat{W}}^{\bar{\tilde{W}}}: \hat{W} \rightarrow
\bar{\tilde{W}}$, which satisfies that for any $k=1, \cdots, N_f$
\begin{eqnarray}\label{ZJ-Basis-BarTildeW-HatW-Tran}
T_{\hat{W}}^{\bar{\tilde{W}}} \phi^{k}_{\Pi, l}=
\phi^{k}_{\Pi, l},~l=1,\cdots,
n_{\Pi}^k;~T_{\hat{W}}^{\bar{\tilde{W}}} \phi^{k}_{\Delta, l}=
\phi^{k,r}_{\Delta,l}
+\phi^{k,j}_{\Delta,l},~l=1,\cdots, n_{\Delta}^k,
\end{eqnarray}
and for any $k=1,\cdots,N_v$
\begin{eqnarray}\label{ZJ-Basis-BarTildeW-HatW-Tran-c}
T_{\hat{W}}^{\bar{\tilde{W}}} \psi^{k}_{l}=
\psi^{k}_{l},~l=1,\cdots,n_k^c,
\end{eqnarray}
where $\{\phi^{k}_{\Delta, l}\}$ and $\{\phi^{k}_{\Pi, l}\}$ are defined in \eqref{ZJ-Def-Bases-WkDel-WkPi}, and $\{\psi^k_l\}$ are defined in \eqref{def-Phik-v}.

From \eqref{ZJ-Basis-BarTildeW-HatW-Tran} and the definition
\eqref{def-D-operator-matrix} of the scaling operators
$D_{F_k}^{(r)}$ $(k \in \mathcal{M}_r, 1 \le r \le N_d)$, we can easily prove
that
\begin{eqnarray}\label{property-I-D}
T_{\hat{W}}^{\bar{\tilde{W}}} T_{W_{k}^{(r)}}^{W_k}  D_{F_k}^{(r)} w
= D_{F_k}^{(r)}(w + T^{W_k^{(j)}}_{W_k^{(r)}} w),~\forall w
\in W_{k,\Delta}^{(r)},
\end{eqnarray}
here we have used the assumption that $F_k = \partial D_r \cap \partial D_j$.

It follows from \eqref{def-operator-hat-S}, \eqref{def-operator-tilde-A}, \eqref{ZJ-Basis-BarTildeW-HatW-Tran} and \eqref{ZJ-Basis-BarTildeW-HatW-Tran-c} that
\begin{eqnarray}\label{ZJn-Relat-hatS-tildeS}
 \hat{S} = (T_{\hat{W}}^{\bar{\tilde{W}}})^H \tilde{S}
 T_{\hat{W}}^{\bar{\tilde{W}}}.
\end{eqnarray}

Combining \eqref{ZJ-Expression-MBDDC-Inv} and \eqref{ZJn-Relat-hatS-tildeS}, we can obtain the preconditioned operator as
\begin{eqnarray}\label{ZJ-oper-hatG}
\hat{G} =M_{BDDC}^{-1} \hat{S} = E_D \tilde{S}^{-1} (E_D)^H
(T_{\hat{W}}^{\bar{\tilde{W}}})^H \tilde{S}
T_{\hat{W}}^{\bar{\tilde{W}}}.
\end{eqnarray}

In the following, we will derive the upper bound of the condition number of $\hat{G}$.

For any $\tilde{w} \in \tilde{W}$, using \eqref{ZJ-TildeW-FuncDecomp}, we have
\begin{eqnarray}\label{def-PD-tildew}
\tilde{w} = \sum\limits_{r=1}^{N_d} \sum\limits_{k \in \mathcal{M}_r} w_{k,\Delta}^{(r)} + w_{\Pi},~~w_{\Pi} := \sum\limits_{k=1}^{N_f} w_{k,\Pi} + \sum\limits_{k=1}^{N_v} w_{k,c},
\end{eqnarray}
where
\begin{eqnarray}\label{def-PD-tildew-para}
w_{k,\Delta}^{(r)} = (\vec{w}_{k,\Delta}^{(r)})^T \Phi_{\Delta}^{k,r} \in W_{k,\Delta}^{(r)},~~w_{k,\Pi} = (\vec{w}_{k,\Pi})^T \Phi_{\Pi}^{k} \in W_{k,\Pi},~~w_{k,c} = (\vec{w}_{k,c})^T \Psi_{\Pi}^{k} \in W_{k,c},
\end{eqnarray}
here $\vec{w}_{k,\Delta}^{(r)} \in \mathbb{C}^{n_{\Delta}^k}$, $\vec{w}_{k,\Pi} \in \mathbb{C}^{n_{\Pi}^k}$ and $\vec{w}_{k,c} \in \mathbb{C}^{n_k^c}$.

Following Theorem 1 in \cite{LW2006}, we need to estimate the bound
\begin{align*}
\tilde{a}(P_D \tilde{w}, P_D \tilde{w}) \le C \tilde{a}(\tilde{w}, \tilde{w}),~\forall \tilde{w} \in \tilde{W},
\end{align*}
where $P_D: \tilde{W} \rightarrow \tilde{W}$ is a jump operator defined as
\begin{eqnarray}\label{def-Pd-0}
P_D = I_{\tilde{W}}^{\tilde{W}} - T_{\hat{W}}^{\bar{\tilde{W}}} E_D,
\end{eqnarray}
here $I_{\tilde{W}}^{\tilde{W}}$ is the identity operator on $\tilde{W}$, the linear basis transformation operator $T_{\hat{W}}^{\bar{\tilde{W}}}$ and the average operator $E_D$ are separately defined in \eqref{ZJ-Basis-BarTildeW-HatW-Tran} and \eqref{ZJ-Def-ED}.

\begin{lemma}\label{lemma-PDw}
For any $\tilde{w} \in \tilde{W}$, we have
\begin{eqnarray}\label{exp-Pdwi}
P_D \tilde{w}
= \sum\limits_{r=1}^{N_d} \sum\limits_{k \in \mathcal{M}_r}
(w_{D,\Delta}^{k,r} - \tilde{w}_{D,\Delta}^{k,r}),
\end{eqnarray}
where
\begin{eqnarray}\label{exp-Pdwi-para}
w_{D,\Delta}^{k,r}:= (\vec{w}_{k,\Delta}^{(r)})^T
\Phi_{D,\Delta}^{k,r},~~\tilde{w}_{D,\Delta}^{k,r} :=
(\vec{w}_{k,\Delta}^{(j)})^T \Phi_{D,\Delta}^{k,r},
\end{eqnarray}
here the basis function vector $\Phi_{D,\Delta}^{k,r}$ is defined in \eqref{ZJ-TDphik-ij-Def}.
\end{lemma}

\begin{proof}
By the definition \eqref{def-Pd-0} of $P_D$, \eqref{ZJ-Def-ED} of $E_D$, \eqref{ZJ-Basis-BarTildeW-HatW-Tran} of $T_{\hat{W}}^{\bar{\tilde{W}}}$, and using the decomposition \eqref{def-PD-tildew},
$P_D \tilde{w}$ can be rewritten as follows:
\begin{eqnarray}\label{lemmaLT-I-equ-proof-1-001-0}\nonumber
P_D \tilde{w}
&=&
\tilde{w} - T_{\hat{W}}^{\bar{\tilde{W}}}E_D \tilde{w}
\\\nonumber
&=&\tilde{w} - T_{\hat{W}}^{\bar{\tilde{W}}}(Q_D +  R_{\tilde{W}}^{W_{\Pi}}) \tilde{w}
\\\nonumber
&=& (\sum\limits_{r=1}^{N_d} \sum\limits_{k \in \mathcal{M}_r} w_{k,\Delta}^{(r)} + w_{\Pi}) - (T_{\hat{W}}^{\bar{\tilde{W}}}Q_D \tilde{w} + w_{\Pi})
\\
&=& \sum\limits_{r=1}^{N_d} \sum\limits_{k \in \mathcal{M}_r} w_{k,\Delta}^{(r)}
   - T_{\hat{W}}^{\bar{\tilde{W}}}Q_D \tilde{w}.
\end{eqnarray}

By using \eqref{ZJ-Def-QD}, \eqref{def-PD-tildew}, \eqref{def-R-property-1}
 and \eqref{property-I-D}, we find the second term in
\eqref{lemmaLT-I-equ-proof-1-001-0} satisfies~($F_k = \partial D_r \cap \partial D_j$)
\begin{eqnarray}\label{equ-IR-1} \nonumber
T_{\hat{W}}^{\bar{\tilde{W}}} Q_D \tilde{w}
&=& T_{\hat{W}}^{\bar{\tilde{W}}}
\sum\limits_{r=1}^{N_d} \hat{I}_{W_{\Delta}^{(r)}}^D \sum\limits_{k
\in \mathcal{M}_r} w_{k,\Delta}^{(r)}\\\nonumber
&=& \sum\limits_{r=1}^{N_d} \sum\limits_{k \in \mathcal{M}_r}
T_{\hat{W}}^{\bar{\tilde{W}}}  T^{W_k}_{W_k^{(r)}} D_{F_k}^{(r)}
w_{k,\Delta}^{(r)}\\
&=& \sum\limits_{r=1}^{N_d} \sum\limits_{k \in \mathcal{M}_r}
D_{F_k}^{(r)}(w_{k,\Delta}^{(r)} + T_{W_k^{(r)}}^{W_k^{(j)}}
w_{k,\Delta}^{(r)}).
\end{eqnarray}

Combining \eqref{lemmaLT-I-equ-proof-1-001-0} and \eqref{equ-IR-1},  and using
\eqref{def-D-trans}, \eqref{def-PD-tildew-para}, \eqref{ZJ-Basis-WkiToWk-Tran-ij},
\eqref{Zj-Def-Vec-TildPhiDel-1}, \eqref{ZJ-Dphik-ij-Def} and
\eqref{ZJ-TDphik-ij-Def}, we have
\allowdisplaybreaks
\begin{eqnarray*}
P_D \tilde{w}
&=& \sum\limits_{r=1}^{N_d} \sum\limits_{k \in
\mathcal{M}_r} w_{k,\Delta}^{(r)} - \sum\limits_{r=1}^{N_d}
\sum\limits_{k \in \mathcal{M}_r} D_{F_k}^{(r)}(w_{k,\Delta}^{(r)}
+ T_{W_k^{(r)}}^{W_k^{(j)}} w_{k,\Delta}^{(r)})\\\nonumber
&=& \sum\limits_{r=1}^{N_d} \sum\limits_{k \in \mathcal{M}_r}
(D_{F_k}^{(j)}w_{k,\Delta}^{(r)} - D_{F_k}^{(r)}
T_{W_k^{(r)}}^{W_k^{(j)}} w_{k,\Delta}^{(r)})\\\nonumber
&=& \sum\limits_{k=1}^{N_f} \left((D_{F_k}^{(j)}
w_{k,\Delta}^{(r)} - D_{F_k}^{(r)} T_{W_k^{(r)}}^{W_k^{(j)}}
w_{k,\Delta}^{(r)}) + (D_{F_k}^{(r)} w_{k,\Delta}^{(j)} -
D_{F_k}^{(j)} T_{W_k^{(j)}}^{W_k^{(r)}}
w_{k,\Delta}^{(j)})\right)\\\nonumber
&=& \sum\limits_{r=1}^{N_d} \sum\limits_{k \in
\mathcal{M}_r}D_{F_k}^{(j)}(w_{k,\Delta}^{(r)} -
T_{W_k^{(j)}}^{W_k^{(r)}} w_{k,\Delta}^{(j)})\\
&=& \sum\limits_{r=1}^{N_d} \sum\limits_{k \in \mathcal{M}_r}
(w_{D,\Delta}^{k,r} - \tilde{w}_{D,\Delta}^{k,r})
\end{eqnarray*}
where $w_{D,\Delta}^{k,r}$ and $\tilde{w}_{D,\Delta}^{k,r}$ are defined in \eqref{exp-Pdwi-para}. That is to say \eqref{exp-Pdwi} holds.
\end{proof}

\begin{lemma}\label{lemma-max}
For a given tolerance $\Theta \ge 1$, we can obtain the following estimate for $\tilde{w} \in \tilde{W}$
\begin{align*}
\tilde{a}(P_D \tilde{w}, P_D \tilde{w}) \le C \Theta \tilde{a}(\tilde{w}, \tilde{w}),
\end{align*}
where $C = 2 C_F^2$ and $C_F = \max\limits_{r}\{f_r\}$, here $f_r$ denotes the number of interface on $\partial D_r$. 
\end{lemma}

\begin{proof}
In view of the definition \eqref{def-operator-tilde-A} of the sesquilinear form $\tilde{a}(\cdot,\cdot)$, \eqref{def-normAi} of the semi-norm $|\cdot|_{a_r}$, and the decompositions \eqref{def-PD-tildew} and \eqref{exp-Pdwi}, it is equivalent to show that
\begin{eqnarray}\label{Pb-max-eigenvalue-estimate-equiv-1-000-a}
\sum\limits_{r=1}^{N_d} |(P_D \tilde{w})^{(r)}|_{a_r}^2 \le C
\Theta \sum\limits_{r=1}^{N_d} |\tilde{w}^{(r)}|_{a_r}^2,
\end{eqnarray}
where
\begin{eqnarray}\label{PDi}
\tilde{w}^{(r)} = \sum\limits_{k \in \mathcal{M}_r} (w_{k,\Delta}^{(r)} + w_{k,\Pi}^{(r)}) + \sum\limits_{k\in \mathcal{M}_r^c} w_{k,c}^{(r)},~~(P_D \tilde{w})^{(r)} =  \sum\limits_{k \in \mathcal{M}_r} (w_{D,\Delta}^{k,r} - \tilde{w}_{D,\Delta}^{k,r}),
\end{eqnarray}
here
\begin{align*}
w_{k,\Pi}^{(r)} = (\vec{w}_{k,\Pi})^T \Phi_{\Pi}^{k,r} \in W_{k,\Pi}^{(r)},~~w_{k,c}^{(r)} = (\vec{w}_{k,c})^T \Psi^{k,r} \in W_{k,c}^{(r)}.
\end{align*}

By \eqref{PDi},
$C_F = \max\limits_{r}\{f_r\}$,
the conditions  \eqref{ZJ-Tphik-TDphik-Prop} associate with $T_{F_k} (k\in \mathcal{M}_r, 1\le r\le N_d)$, and  \eqref{4-1-2-tildew}, we have 
\allowdisplaybreaks
\begin{eqnarray}\label{final-ineq-1}\nonumber
\sum\limits_{r=1}^{N_d} |(P_D\tilde{w})^{(r)}|_{a_r}^2
&=&\sum\limits_{r=1}^{N_d} |\sum\limits_{k \in
\mathcal{M}_r}(w_{D,\Delta}^{k,r} -
\tilde{w}_{D,\Delta}^{k,r})|_{a_r}^{2}\\\nonumber
&\le& 2 C_F
\sum\limits_{r=1}^{N_d}\sum\limits_{k \in
\mathcal{M}_r}\left(|w_{D,\Delta}^{k,r}|_{a_r}^2 +
|\tilde{w}_{D,\Delta}^{k,r}|_{a_r}^2 \right)\\\nonumber
&=& 2 C_F \sum\limits_{r=1}^{N_d} \sum\limits_{k \in \mathcal{M}_r}
\left(|w_{D,\Delta}^{k,r}|_{a_r}^2 +
|\tilde{w}_{D,\Delta}^{k,j}|_{a_j}^2\right)\\ \nonumber
&\le& 2 C_F \Theta \sum\limits_{r=1}^{N_d} \sum\limits_{k \in
\mathcal{M}_r} |\bar{w}_{k,\Delta}^{(r)} +
\bar{w}_{k,\Pi}^{(r)}|_{a_r}^2\\
&\le& 2 C_F^2 \Theta \sum\limits_{r=1}^{N_d}
|\tilde{w}^{(r)}|_{a_r}^2
\end{eqnarray}
where
\begin{align*}
\tilde{w}_{D,\Delta}^{k,j} := (\vec{w}_{k,\Delta}^{(r)})^T \Phi_{D,\Delta}^{k,j},~~
\bar{w}_{k,\Delta}^{(r)} = (\vec{w}_{k,\Delta}^{(r)})^T \bar{\Phi}_{\Delta}^{k,r} \in \bar{W}_{k,\Delta}^{(r)},~~
\bar{w}_{k,\Pi}^{(r)}  = (\vec{w}_{k,\Pi})^T \bar{\Phi}_{\Pi}^{k,r}  \in \bar{W}_{k,\Pi}^{(r)}.
\end{align*}

Finally, the estimate \eqref{Pb-max-eigenvalue-estimate-equiv-1-000-a} follows from \eqref{final-ineq-1}.
\end{proof}

By Lemma \ref{lemma-max} and following Theorem 1 in \cite{LW2006},  we obtain:
\begin{theorem}\label{condition-number}
For a given tolerance $\Theta \ge 1$, the condition number bound of the two-level adaptive BDDC preconditioned systems $\hat{G}$
\begin{align*}
\kappa(\hat{G}) \le C \Theta,
\end{align*}
where  $C$ is a constant which is just depending on the maximum number of interfaces per subdomain.
\end{theorem}

\section{Numerical results}\setcounter{equation}{0}

In this section, we present examples with constant and variable wave number $\kappa$ to investigate the convergence properties of the adaptive BDDC preconditioners proposed in this paper.

In the following numerical tests, the given region $\Omega$ is divided into $N_h$ quadrilateral mesh $\mathcal{T}_h$,
where $h$ denotes the size of elements.
Let $p$ denote the number of plane wave bases in each element. Therefore, it is easy to known that the number of dofs
in the plane wave space $V_p(\mathcal{T}_h)$ is $N_h \times p$.
Since the special interface is needed in this paper, and to guarantee the load balance,
we decompose $\mathcal{T}_h$ into some subdomains which satisfy that the number of complete elements in each subdomain is the same and denotes as $n$.
For example if we set $N_h = 35^2$ and the number of subdomains $N_d = 4^2$, then $n = 8^2$.

Since the stiffness matrix of the PWLS method is HPD,
we solve the Schur complement system \eqref{def-schur-system} by PCG method, and the iteration is stopped either the relative residual is reduced by the factor of $10^{-5}$ or the iteration counts are greater than 100. These algorithms are implemented using Matlab
and run in a machine with Intel(R) Xeon(R) CPU E5-2650 v2 2.60 GHz and 96 GB memory.

Numerical experiment results show that the modulus of the eigenvalue $\lambda$ of the generalized eigenvalue problems \eqref{eig-pro-intro} satisfies $|\lambda| \ge 1$, hence, in our adaptive BDDC algorithms, we set the tolerance $\Theta = 1+ log(\min\limits_{1 \le r \le N_d}\{n_x^{(r)}, n_y^{(r)}\})$ for a given mesh partition,
where $n_x^{(r)}, n_y^{(r)}$ are separately the number of complete and part elements in $x$ and $y$ direction of subdomain $D^{(r)} (r=1,\cdots,N_d)$.
The transform matrix $\vec{T}_{F_k}$ in each interface $F_k$ is defined in \eqref{ZJ-Oper-TFk-DelPi-Def}.
Therefore, the algorithm is uniquely determined by the scaling matrices $\vec{D}_{F_k}^{(\nu)}(\nu=r,j)$ for each interface.
In the following experiments, we separately denote the algorithm with $\vec{D}_{F_k}^{(\nu)}(\nu=r,j, k=1,\cdots,N_f)$ defined in \eqref{DiFD} and \eqref{DiFD-deluxe} as method1 and method2.
We will apply both methods to three typical examples to investigate the scalability of these methods
measured by the mesh size, number of subdomains and angular frequency (or wave number).
Since the wave number $\kappa = \omega/c$, where $\omega$ and $c$ are separately the angular frequency and the wave speed,
we can react the variousness of the wave number to the angular frequency and the wave speed.

In all of the tables in this section,  iter is the number of iterations for the PCG algorithms, $\lambda_{\min}$ and $\lambda_{\max}$ separately denote the minimum  and maximum eigenvalues of the preconditioned system,  pnum is the number of primal unknowns, the average number of primal unknowns per interface are given in the parentheses, and the proportion of the total number of primal unknowns to the total number of interface dofs are denoted as ppnum.

Before numerical studies are performed, three typical examples are given here.

\begin{example}{\bf (Constant medium)}\label{Example-1}
\cite{HZ2016} Consider model problem \eqref{model equation}, where $\Omega = (0,2) \times (0,1)$, the wave speed $c \equiv 1$, the exact solution of the problem can be expressed as
\begin{align*}
u_{ex} = \cos(12 \pi y)(A_1 e^{-i \omega_x x} + A_2 e^{i \omega_x x}),
\end{align*}
here $\omega_x = \sqrt{\omega^2 - (12 \pi)^2}$, and coefficients $A_1$ and $A_2$ satisfy the equation
\begin{align*}
& \left(
\begin{array}{cc}
\omega_x  & -\omega_x \\
(\omega - \omega_x)e^{-2i \omega_x} & (\omega + \omega_x)e^{2i \omega_x}
\end{array}
\right) \left(\begin{array}{c}
A_1 \\
A_2
\end{array}\right)
= \left(\begin{array}{c}
-i\\
0
\end{array}\right).
\end{align*}
\end{example}

\begin{example}{\bf (Piecewise constant medium)}\label{Example-2}
Consider model problem \eqref{model equation}, which is a variant of the Marmousi model in \cite{S2013} or \cite{HZ2016},
where $\Omega = (0,7200)\times(0,3600)$, and the wave speed $c$ is defined by (see figure \ref{fig-piecewise-constant} for illustration)
\begin{equation*}
c(x,y) = \left\{\begin{array}{ll}
1800 & \mbox{if}~y \in [0,1200]\\
3600 & \mbox{if}~y \in [1200,2400]\\
5400 & \mbox{if}~y \in [2400,3600]
\end{array}\right.~~x\in [0,7200],
\end{equation*}
and $g = x^2 + y^2$.
\begin{figure}[H]
  \centering
  \includegraphics[width=0.7\textwidth]{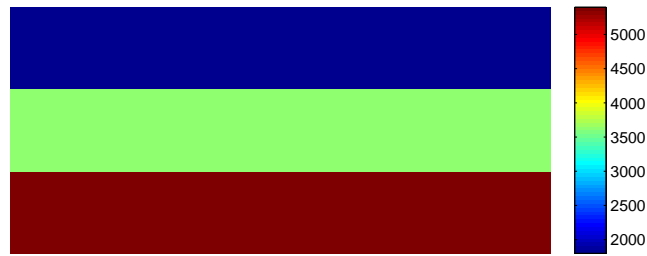}\\
  \caption{Piecewise constant medium.}\label{fig-piecewise-constant}
\end{figure}
\end{example}

\begin{example}{\bf (Random medium)}\label{Example-3}
Consider Example \ref{Example-2}, where $c(x,y)$  is chosen randomly from $[1500,5500]$ for each grid element, as shown in Figure \ref{fig-random},
which can be seen as a complex version of the Marmousi model in \cite{S2013}.
\begin{figure}[H]
  \centering
  \includegraphics[width=0.7\textwidth]{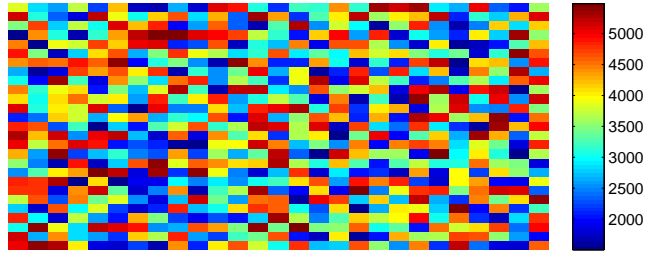}\\
  \caption{Random medium.}\label{fig-random}
\end{figure}
\end{example}

\subsection{Two-level results}

\subsubsection{Scalability study with respect to angular frequency}

In this subsection, we investigate the influence of the angular frequency (or the wave number) to the efficiency of our adaptive BDDC algorithms.
This numerical studies are carried on a $N_d = 4^2$ subdomain partition with different frequency $\omega$.
The relative $L^2$ error presented in Table \ref{E-table-1} is defined as
\begin{align*}
\frac{\|u_{ex} - u_h\|_{L^2(\Omega)}}{\|u_{ex}\|_{L^2(\Omega)}},
\end{align*}
where $u_{ex}$ and $u_h$ are separately the exact solution and the numerical solution of this problem.
We kept $\omega h$ as a constant and slightly increase $p$ or decrease $h$ to control the relative error less than $5\times 10^{-3}$.
In this article, since the errors of the approximate solutions are not our main interest, we only list the relative $L^2$ errors in Table \ref{E-table-1}
for Example \ref{Example-1}.

As seen in Table \ref{E-table-1}, with the same stopping criterion, the relative $L^2$ errors of the approximate solutions of Example \ref{Example-1} generated by method1 and method2 do not have large differences.
\begin{table}[H]
\centering\caption{
Relative error of Example \ref{Example-1} solved by method1 and method2.
}
\label{E-table-1}\vskip 0.1cm
\begin{tabular}{{|c|c|c|c|c|}}\hline
\multirow{2}{*}{$\omega$}& \multirow{2}{*}{$p$}    &\multirow{2}{*}{$N_h$}    & \multicolumn{2}{c|}{method} \\\cline{4-5}
                         &                         &                          &  method1          & method2 \\\hline
$20\pi$                  & $13$                    & $35^2$                   &  1.160e-04  &1.163e-04\\\hline
$40\pi$                  & $13$                    & $63^2$                   &  9.293e-04  &9.266e-04\\\hline
$80\pi$                  & $14$                    & $119^2$                  &  9.952e-04  &9.910e-04\\\hline
$160\pi$                 & $14$                    & $255^2$                  &  5.804e-04  &5.793e-04\\\hline
$320\pi$                 & $15$                    & $511^2$                  &  3.091e-03  &3.091e-03\\\hline
\end{tabular}
\end{table}

The dependence of the iteration counts on the angular frequency is presented in Table \ref{E-table-2}, Table \ref{E-table-3} and Table \ref{E-table-4} for various examples, and a weak dependency relationship is clearly shown.
In addition, from these tables, we can see that the total number of primal unknowns (pnum) increased with the increase of the angular frequency for both method1 and method2, especially when $\omega = 320\pi$ in Example \ref{Example-1}, the dofs of the coarse problem has reached more than thirty thousand in method1.
The number of primal unknowns also highlights the superiority of method2
over method1.

\begin{table}[H]
\centering\caption{
Scalability study with respect to the angular frequency: Example \ref{Example-1}.
}
\label{E-table-2}\vskip 0.1cm
\begin{tabular}{{|c|c|c|c|c|c|c|c|c|}}\hline
$\omega$                 & $p$                     &$N_h$                   & method  & $\lambda_{\min}$ & $\lambda_{\max}$ & pnum &ppnum  & iter \\\hline
\multirow{2}{*}{$20\pi$} & \multirow{2}{*}{$13$}   &\multirow{2}{*}{$35^2$}   &method1   &1.0003  &2.3717   &2019(84.13)  &81.75\%    &7  \\
                         &                         &                          &method2   &1.0003  &3.3375   &639(26.63)   &28.93\%    &8  \\\hline
\multirow{2}{*}{$40\pi$} & \multirow{2}{*}{$13$}   &\multirow{2}{*}{$63^2$}   &method1   &1.0002  &2.3538   &3672(153.00) &78.99\%    &7  \\
                         &                           &                        &method2   &1.0002  &2.0633   &995(39.79)   &22.35\%    &6  \\\hline
\multirow{2}{*}{$80\pi$} & \multirow{2}{*}{$14$}   &\multirow{2}{*}{$119^2$}  &method1   &1.0000  &2.3496   &7710(321.25) &79.39\%    &6  \\
                         &                         &                          &method2   &1.0002  &4.2066   &1608(67.00)  &17.57\%    &8  \\\hline
\multirow{2}{*}{$160\pi$} & \multirow{2}{*}{$14$}  &\multirow{2}{*}{$255^2$}  &method1   &1.0000  &2.3582   &16686(695.25)&78.95\%    &5  \\
                         &                         &                          &method2   &1.0001  &5.0005   &2869(119.54) &14.06\%    &11 \\\hline
\multirow{2}{*}{$320\pi$} & \multirow{2}{*}{$15$}  &\multirow{2}{*}{$511^2$}  &method1   &1.0000  &2.7400   &36630(1526.25) & 80.18\%    &6   \\
                         &                         &                          &method2   &1.0000  &3.6244   &5444(226.83) &12.17\%    &9 \\\hline
\end{tabular}
\end{table}


\begin{table}[H]
\centering\caption{Scalability study with respect to the angular frequency: Example \ref{Example-2}}
\label{E-table-3}\vskip 0.1cm
\begin{tabular}{{|c|c|c|c|c|c|c|c|c|}}\hline
$\omega$                 & $p$                     &$N_h$                   & method  & $\lambda_{\min}$ & $\lambda_{\max}$ & pnum &ppnum  & iter\\\hline
\multirow{2}{*}{$20\pi$} & \multirow{2}{*}{$13$}   &\multirow{2}{*}{$35^2$}   &method1   &1.0003  &3.6775   &1997(83.21)  &80.00\%    &9  \\
                         &                         &                        &method2   &1.0000  &5.1027   &653(27.12)   &26.16\%    &7  \\\hline
\multirow{2}{*}{$40\pi$} & \multirow{2}{*}{$13$}   &\multirow{2}{*}{$63^2$}   &method1   &1.0000  &3.2890   &3526(146.92) &75.34\%    &7  \\
                         &                         &                        &method2   &1.0000  &4.1102   &868(36.17)   &18.55\%    &7  \\\hline
\multirow{2}{*}{$80\pi$} & \multirow{2}{*}{$13$}   &\multirow{2}{*}{$119^2$}  &method1   &1.0000  &3.8117   &6389(266.21) &79.99\%    &8 \\
                         &                         &                        &method2   &1.0000  &4.5624   &1165(48.54)  &13.99\%    &9  \\\hline
\multirow{2}{*}{$160\pi$} & \multirow{2}{*}{$13$}  &\multirow{2}{*}{$255^2$}  &method1   &1.0000  &4.5980   &13629(567.88)&69.52\%    &9 \\
                         &                         &                        &method2   &1.0000  &5.0878   &1807(75.29)  &9.73\%     &10 \\\hline
\end{tabular}
\end{table}

\begin{table}[H]
\centering\caption{Scalability study with respect to the angular frequency: Example \ref{Example-3}}
\label{E-table-4}\vskip 0.1cm
\begin{tabular}{{|c|c|c|c|c|c|c|c|c|}}\hline
$\omega$                 & $p$                     &$N_h$                   & method  & $\lambda_{\min}$ & $\lambda_{\max}$ & pnum &ppnum  & iter\\\hline
\multirow{2}{*}{$20\pi$} & \multirow{2}{*}{$13$}   &\multirow{2}{*}{$35^2$}   &method1   &1.0011  &2.9146   &1896(79.00)  &77.04\%    &9  \\
                         &                         &                        &method2   &1.0027  &2.2602   &737(30.71)   &32.68\%    &8  \\\hline
\multirow{2}{*}{$40\pi$} & \multirow{2}{*}{$13$}   &\multirow{2}{*}{$63^2$}   &method1   &1.0033  &3.5898   &2905(121.04) &63.00\%    &11  \\
                         &                         &                        &method2   &1.0013  &2.5865   &687(28.63)   &16.76\%    &8  \\\hline
\multirow{2}{*}{$80\pi$} & \multirow{2}{*}{$14$}   &\multirow{2}{*}{$119^2$}  &method1   &1.0056  &4.3190   &5459(227.46) &56.59\%    &12 \\
                         &                         &                        &method2   &1.0006  &3.0389   &786(32.75)   &9.24\%     &8  \\\hline
\multirow{2}{*}{$160\pi$} & \multirow{2}{*}{$14$}  &\multirow{2}{*}{$255^2$}  &method1   &1.0030  &5.1083   &10661(444.21)&50.66\%  &13 \\
                         &                         &                        &method2   &1.0002  &3.2627   &801(33.38)  &4.35\%     &9 \\\hline
\end{tabular}
\end{table}

\subsubsection{Scalability study with respect to number of subdomains}

In the second subsection, we set the number of the complete elements in each subdomain $n$ is $8^2$, $p=10$ and
the frequency to $\omega = 20\pi$. The effect of number of the subdomains is explored
here by adding $N_d$ from $3^2$ to $6^2$.

Results of these experiments are shown in Table \ref{E-table-5}, Table \ref{E-table-6}
and Table \ref{E-table-7}. From these tables, we can see that the iteration counts are almost independent on the number of subdomains, and as the number of subdomains increases, the total number of primal unknowns increase, but the average number of primal unknowns per interface almost remains the same.
Apart from this, one can see that for these three examples, method2 still has an absolute advantage in the number of the primal dofs over method1,
in particular, less than half of the primal dofs are required for method2 to ensure the convergence stability.

\begin{table}[H]
\centering\caption{Scalability study with respect to number of subdomains: Example \ref{Example-1}.}
\label{E-table-5}\vskip 0.1cm
\begin{tabular}{{|c|c|c|c|c|c|c|}}\hline
$N_d$                   & method  & $\lambda_{\min}$ & $\lambda_{\max}$ & pnum &ppnum  & iter\\\hline
\multirow{2}{*}{$3^2$}    &method1   &1.0001  &2.4500   &692(57.67)   &73.20\%    &7  \\
                        &method2   &1.0011  &2.6582   &272(22.67)   &31.20\%    &7  \\\hline
\multirow{2}{*}{$4^2$}    &method1   &1.0001  &2.6281   &1416(59.00)  &74.93\%    &7  \\
                        &method2   &1.0007  &1.9983   &554(23.08)   &32.04\%    &6  \\\hline
\multirow{2}{*}{$5^2$}    &method1   &1.0002  &2.7296   &2368(59.20)  &75.24\%    &7 \\
                        &method2   &1.0007  &2.5805   &885(22.13)   &31.10\%    &8  \\\hline
 \multirow{2}{*}{$6^2$}    &method1   &1.0001  &2.8044   &3560(59.33)  &75.45\%    &7 \\
                        &method2   &1.0008  &2.6865   &1274(21.23)  &30.18\%    &8 \\\hline
\end{tabular}
\end{table}

\begin{table}[H]
\centering\caption{Scalability study with respect to number of subdomains: Example \ref{Example-2}.}
\label{E-table-6}\vskip 0.1cm
\begin{tabular}{{|c|c|c|c|c|c|c|}}\hline
$N_d$                   & method  & $\lambda_{\min}$ & $\lambda_{\max}$ & pnum &ppnum  & iter\\\hline
\multirow{2}{*}{$3^2$}    &method1   &1.0003  &2.9381   &639(53.25)   &67.90\%    &8  \\
                        &method2   &1.0000  &2.5374   &244(20.33)   &28.40\%    &7  \\\hline
\multirow{2}{*}{$4^2$}    &method1   &1.0001  &2.7952   &1345(56.04)  &71.39\%    &7  \\
                        &method2   &1.0000  &2.4366   &487(20.29)   &28.71\%    &6  \\\hline
\multirow{2}{*}{$5^2$}    &method1   &1.0001  &2.8565   &2332(58.30)  &74.17\%    &7 \\
                        &method2   &1.0000  &2.9097   &867(21.68)   &30.57\%    &6  \\\hline
 \multirow{2}{*}{$6^2$}   &method1   &1.0002  &2.8892   &3575(59.58)  &75.74\%    &7 \\
                        &method2   &1.0000  &2.9908   &1337(22.28)  &31.43\%    &7 \\\hline
\end{tabular}
\end{table}

\begin{table}[H]
\centering\caption{Scalability study with respect to number of subdomains: Example \ref{Example-3}.}
\label{E-table-7}\vskip 0.1cm
\begin{tabular}{{|c|c|c|c|c|c|c|}}\hline
$N_d$                   & method  & $\lambda_{\min}$ & $\lambda_{\max}$ & pnum &ppnum  & iter\\\hline
\multirow{2}{*}{$3^2$}    &method1   &1.0017  &3.0289   &549(45.74)   &58.90\%    &9  \\
                        &method2   &1.0017  &2.0701   &187(15.58)   &22.70\%    &7  \\\hline
\multirow{2}{*}{$4^2$}    &method1   &1.0015  &3.0633   &1231(51.29)  &65.72\%    &10  \\
                        &method2   &1.0021  &2.2055   &452(18.83)   &26.97\%    &7  \\\hline
\multirow{2}{*}{$5^2$}    &method1   &1.0012  &3.0412   &2204(55.10)  &70.36\%    &10 \\
                        &method2   &1.0024  &2.3148   &821(20.53)   &29.20\%    &8  \\\hline
 \multirow{2}{*}{$6^2$}   &method1   &1.0013  &3.0357   &3461(57.63)  &73.49\%    &9 \\
                        &method2   &1.0020  &2.3250   &1348(22.47)  &31.64\%    &8 \\\hline
\end{tabular}
\end{table}

\subsubsection{Scalability study with respect to mesh size}

Finally, in this subsection, we consider the scalability with respect to mesh size, where $p=9$, the number of the subdomains and the angular frequency are both held constant ($N_d = 4^2$, $\omega = 20\pi$), but the number of complete elements in each subdomain varied from $n=6^2$ to $n=24^2$.

The iteration efficiency as a function of the mesh size are reported in Table \ref{E-table-8}, Table \ref{E-table-9} and Table \ref{E-table-10} for method1 and method2. We can see from these tables that the iteration counts are mildly dependent on the mesh size. The minimum eigenvalues of the preconditioned systems are larger than 1 and the maximum eigenvalues are mildly dependent on the mesh size.
As the mesh size decrease, there is a great difference between method1 and method2 in the number of the primal unknowns even for constant medium, which is different from the adaptive BDDC algorithms for the two-order elliptic problems.


\begin{table}[H]
\centering\caption{Scalability study with respect to mesh size: Example \ref{Example-1}} 
\label{E-table-8}\vskip 0.1cm
\begin{tabular}{{|c|c|c|c|c|c|c|c|}}\hline
$n$                     & $\Theta$                 & method   &$\lambda_{\min}$ & $\lambda_{\max}$ & pnum &ppnum  & iter\\\hline
\multirow{2}{*}{$6^2$}  & \multirow{2}{*}{2.95}  &method1         &1.0004  &2.7845   &984(41.00)   &77.34\%    &8  \\
                        &                          &method2         &1.0008  &1.6270   &438(18.25)   &37.69\%    &6  \\\hline
\multirow{2}{*}{$12^2$} & \multirow{2}{*}{3.56}  &method1         &1.0004  &3.0901   &1815(75.63)  &70.93\%    &9  \\
                        &                          &method2         &1.0010  &1.5878   &535(22.29)   &23.05\%    &5  \\\hline
\multirow{2}{*}{$18^2$} & \multirow{2}{*}{3.94}  &method1         &1.0003  &3.1196   &2664(111.00) &69.16\%    &9 \\
                        &                          &method2         &1.0005  &2.2794   &590(24.58)   &16.91\%    &6  \\\hline
 \multirow{2}{*}{$24^2$}& \multirow{2}{*}{4.22}  &method1         &1.0004  &3.0915   &3528(147.00) &68.55\%    &9 \\
                        &                          &method2         &1.0004  &3.2986   &624(26.00)   &13.39\%    &8 \\\hline
\end{tabular}
\end{table}


\begin{table}[H]
\centering\caption{Scalability study with respect to mesh size: Example \ref{Example-2}} 
\label{E-table-9}\vskip 0.1cm
\begin{tabular}{{|c|c|c|c|c|c|c|c|}}\hline
$n$                     & $\Theta$                 & method   & $\lambda_{\min}$ & $\lambda_{\max}$ & pnum &ppnum  & iter\\\hline
\multirow{2}{*}{$6^2$}  &  \multirow{2}{*}{2.95}  &method1         &1.0002  &2.7853   &828(34.50)   &66.01\%    &8  \\
                        &                         &method2         &1.0002  &2.0208   &357(14.88)   &31.81\%    &6  \\\hline
\multirow{2}{*}{$12^2$} &  \multirow{2}{*}{3.56}  &method1         &1.0000  &3.1426   &1723(71.79)  &67.49\%    &8  \\
                        &                         &method2         &1.0000  &2.7312   &485(20.21)   &21.17\%    &7  \\\hline
\multirow{2}{*}{$18^2$} &  \multirow{2}{*}{3.94}  &method1         &1.0000  &3.0267   &2641(110.04) &68.58\%    &8 \\
                        &                         &method2         &1.0000  &3.4736   &565(23.54)   &16.28\%    &7  \\\hline
 \multirow{2}{*}{$24^2$}&  \multirow{2}{*}{4.22}  &method1         &1.0000  &3.0288   &3516(146.50) &68.32\%    &8 \\
                        &                         &method2         &1.0000  &3.2067   &607(25.29)   &13.07\%    &7 \\\hline
\end{tabular}
\end{table}


\begin{table}[H]
\centering\caption{Scalability study with respect to mesh size: Example \ref{Example-3}}
\label{E-table-10}\vskip 0.1cm
\begin{tabular}{{|c|c|c|c|c|c|c|c|}}\hline
$n$                     & $\Theta$                 & method   &$\lambda_{\min}$ & $\lambda_{\max}$ & pnum &ppnum  & iter\\\hline
\multirow{2}{*}{$6^2$}  &  \multirow{2}{*}{2.95}  &method1         &1.0013  &2.7596   &750(31.25)   &60.35\%    &9  \\
                        &                         &method2         &1.0014  &1.9304   &339(14.13)   &30.50\%    &7  \\\hline
\multirow{2}{*}{$12^2$} &  \multirow{2}{*}{3.56}  &method1         &1.0029  &3.4462   &1451(60.46)  &57.31\%    &10  \\
                        &                         &method2         &1.0016  &2.3924   &399(16.63)   &17.96\%    &8  \\\hline
\multirow{2}{*}{$18^2$} &  \multirow{2}{*}{3.94}  &method1         &1.0047  &3.8510   &2059(85.79)  &53.92\%    &11 \\
                        &                         &method2         &1.0020  &2.5554   &435(18.13)   &13.00\%    &8  \\\hline
 \multirow{2}{*}{$24^2$}&  \multirow{2}{*}{4.22}  &method1         &1.0032  &4.1637   &2595(108.13) &50.83\%    &12 \\
                        &                         &method2         &1.0014  &2.9632   &445(18.54)   &9.99\%     &9 \\\hline
\end{tabular}
\end{table}

The results listed in the above-mentioned tables have verified
the correctness of the theoretical results,
and indicate that the proposed two-level adaptive BDDC algorithms are very effective for solving Helmholtz equation with large and various wave numbers.
However, one fact worth noticing is that the number of primal unknowns increase as the wave numbers or the number of subdomains increase,
the corresponding coarse problem will become too large and hard to solve directly, and it is more obvious for the multiplicity scaling matrices.

\subsection{Multilevel results}


In this subsection, we try to apply the multilevel adaptive BDDC algorithm with multiplicity scaling matrices to Example \ref{Example-1} to show its performance for large wave numbers.

For this test, the PCG algorithm is stopped either the iteration counts are greater than 100 or the relative residual is reduced by the factor of $10^{-5}$ at level $0$ and $10^{-2}$ at other levels. We consider a higher wave number model and set the angular frequency $\omega = 320\pi$, the number of plane wave bases in each element $p=15$.
In the algorithm, four subdomains at the finer level are treated as a coarser subdomain.

%

Table \ref{E-table-2-multi} shows the performance of our multilevel adaptive BDDC algorithm for different $N_d$ and $n$, where Tdofs is the total number of dofs, Fpnum and Cpnum separately denote the number of dofs at level 1 and the coarsest level.
From this table, we can see that the number of dofs at level 1 (Fpnum) increases as the number of subdomains increases,
and the scale of the coarse problem which need to be compute directly are hard to be accepted when we use a two-level adaptive BDDC algorithm.
With the increase of the number of levels (L), the number of dofs at the coarsest level (Cpnum) is reduced to about 30\%, and the iteration counts remains almost the same. That is to say, the multilevel algorithm is effective for reducing the number of dofs at the coarse problem and it is also efficient for solving large wave number problems.

\begin{table}[H]
\centering\caption{Performance of the multilevel adaptive BDDC algorithm.}
\label{E-table-2-multi}\vskip 0.1cm
\begin{tabular}{{|c|c|c|c|c|c|c|c|c|}}\hline
$N_d$                   &$n$                     & Tdofs                &Fpnum                   & L  &Cpnum     &  $\lambda_{\min}$ & $\lambda_{\max}$ & iter \\\hline
\multirow{2}{*}{$8^2$}  &\multirow{2}{*}{$64^2$} &\multirow{2}{*}{3.9M} &\multirow{2}{*}{85K}   & 3  &36K      &1.0000  &2.7871   &7 \\
                        &                        &                      &                       & 4  &12K      &1.0000  &2.7971   &7 \\\hline
\multirow{2}{*}{$16^2$} &\multirow{2}{*}{$31^2$} &\multirow{2}{*}{3.7M} &\multirow{2}{*}{221K}  & 4  &35K      &1.0000  &2.1836   &6 \\
                        &                        &                      &                       & 5  &11K      &1.0000  &2.1837   &6 \\\hline
\end{tabular}
\end{table}


\section{Conclusions}

In this paper, by introducing some auxiliary spaces, dual-primal basis functions and operators with essential
properties, BDDC algorithms with adaptive primal unknowns are developed and analyzed for the PWLS discritizations of the Helmholtz equations with high and various wave number.
Since the dofs of the PWLS discritization are defined on elements rather than vertices or edges,
we introduce a special ``interface" and the corresponding sesquilinear form for each subdomains,
which is different from the interface in the existing BDDC algorithms.
As the eigenvalues of the local generalized eigenvalue problems are complex, we choose the primal constraints which are formed by the eigenvectors with
their complex modulus of eigenvalues greater than a user-defined tolerance value $\Theta$,
which is used to construct the transformation operators for selecting dual-primal basis functions and control the condition number.
We prove that the condition number of the preconditioned system is bounded above by $C \Theta$.
A multilevel algorithm is attempted to resolve the bottleneck in large scale coarse problem.
Numerical results are presented to verify the robustness and efficiency of the proposed approaches.
Further, we will extend this method to three-dimensional
case. There is fundamentally different from two-dimensional case and three-dimensional case,
such as we need to consider the construction of the generalized eigenvalue problem for each edge which is shared by more than two subdomains.

\section*{Acknowledgments}

The authors would like to thank Professor Qiya Hu for his suggestions in the PWLS discritization.
This work is supported by the National Natural Science Foundation of China (Grant Nos. 11571293, 11201398, 11601462), Hunan Provincial Natural
Science Foundation of China (Grant No. 2016JJ2129), General Project of Hunan Provincial Education Department of China
(Grant No. 17C1527), Open Foundation of Guangdong Provincial Engineering Technology Research Center for Data Science(Grant No. 2016KF07), and Hunan Provincial Civil-Military Integration Industrial Development Project ``Adaptive Multilevel
Solver and Its Application in ICF Numerical Simulation".

\section*{References}
%

\bibliographystyle{elsarticle-num}

\FloatBarrier

\end{document}